\documentclass[a4paper,reqno,10pt]{amsart}

\usepackage{epsf,graphicx,epsfig}
\usepackage{amscd}
\usepackage{amsmath,latexsym,amssymb,amsthm}
\usepackage[nospace,noadjust]{cite}
\usepackage{textcomp}
\usepackage{setspace,cite}
\usepackage{lscape,fancyhdr,fancybox}
\usepackage{stmaryrd}
\usepackage[all,cmtip]{xy}
\usepackage{tikz}
\usepackage{cancel}
\usetikzlibrary{shapes,arrows,decorations.markings}
\usepackage{graphicx}

\usepackage{color}
\usepackage{url}
\usepackage{enumerate}
\usepackage[mathscr]{euscript}

  \theoremstyle{plain}

\swapnumbers
    \newtheorem{thm}{Theorem}[section]
    \newtheorem{prop}[thm]{Proposition}
     
   \newtheorem{lemma}[thm]{Lemma}
    \newtheorem{corollary}[thm]{Corollary}
    
    \newtheorem{subsec}[thm]{}
\theoremstyle{definition}
    \newtheorem{defn}[thm]{Definition}
        \newtheorem{remark}[thm]{Remark}
    \newtheorem{exam}[thm]{Example}

\theoremstyle{remark}

\title{}
\author{}
\date{}
\usepackage{amssymb}

\usepackage{hyperref}
\hypersetup{
	colorlinks,
	citecolor=blue,
	filecolor=black,
	linkcolor=blue,
	urlcolor=black
}

\begin{document}
\title{Leibniz algebras with derivations}

\author{Apurba Das}
\address{Department of Mathematics and Statistics,
Indian Institute of Technology, Kanpur 208016, Uttar Pradesh, India.}
\email{apurbadas348@gmail.com}



\noindent

\thispagestyle{empty}

\maketitle

\begin{abstract}
In this paper, we consider Leibniz algebras with derivations. A pair consisting of a Leibniz
algebra and a distinguished derivation is called a LeibDer pair. We define a cohomology theory
for LeibDer pair with coefficients in a representation. We study central extensions and abelian extensions of a LeibDer
pair. In the next, we generalize the formal deformation theory to
LeibDer pairs in which we deform both the Leibniz bracket and the distinguished derivation. It is governed
by the cohomology of LeibDer pair with coefficients in itself. Finally, we consider homotopy
derivations on sh Leibniz algebras and $2$-derivations on Leibniz $2$-algebras. The category of
$2$-term sh Leibniz algebras with homotopy derivations is equivalent to the category of Leibniz
$2$-algebras with $2$-derivations.
\end{abstract}

\tableofcontents

\vspace{0.2cm}

\section{Introduction}
Leibniz algebras (also called Loday algebras) are noncommutative analog of Lie algebras. They were introduced by Jean-Louis Loday \cite{loday-une}. In the same paper, he introduces a homology theory for Leibniz algebras, a noncommutative generalization of the Lie algebra homology. In \cite{loday-pira} Loday and Pirashvili introduced a cohomology theory for Leibniz algebras with coefficients in a representation. This cohomology is also a noncommutative analog of the Chevalley-Eilenberg cohomology for Lie algebras. The classical deformation theory of Gerstenhaber \cite{gers} has been extended to Leibniz algebras in \cite{bala-leib}. The notion of sh Leibniz algebras (strongly homotopy Leibniz algebras) are introduced in \cite{ammar-poncin}. Strongly homotopy Leibniz algebras are related to categorification of Leibniz algebras \cite{sheng-liu}. See \cite{loday-cup,branes,ayupov} for some interesting results about Leibniz algebras.

\medskip

Algebraic structures are useful via their derivations. Derivations are also useful in constructing
homotopy Lie algebras \cite{voro}, deformation formulas \cite{coll} and differential Galois theory \cite{magid}. They also play an essential role in control theory and gauge theories in quantum field theory \cite{ayala-1, ayala-2}. In \cite{doubek-lada,loday} the authors study algebras with derivations from operadic point of view. Recently, Lie algebras with derivations (called LieDer pairs) are studied from cohomological point of view \cite{tang} and extensions, deformations of LieDer pairs are considered. The results of \cite{tang} has been extended to associative algebras with derivations (called AssDer pairs) in \cite{das-mandal}.

\medskip

In this paper, we consider Leibniz algebras with derivations. More precisely, we consider a pair
$(\mathfrak{g}, \phi_\mathfrak{g})$ where $\mathfrak{g}$ is a Leibniz algebra and $\phi_\mathfrak{g} : \mathfrak{g} \rightarrow \mathfrak{g}$  is a derivation for the Leibniz algebra bracket on $\mathfrak{g}$. We call such a pair $(\mathfrak{g}, \phi_\mathfrak{g})$ a LeibDer pair. We define representations and cohomology for a LeibDer pair. This cohomology is a variant of the Leibniz algebra cohomology. When considering the cohomology of a LeibDer pair $(\mathfrak{g}, \phi_\mathfrak{g})$ with coefficients in itself, we show that the cohomology inherits a degree $-1$ graded Lie bracket.

\medskip

Central extensions and abelian extensions of Leibniz algebras was defined in \cite{loday-pira}. They are
related to the second cohomology group of Leibniz algebras with suitable coefficients. In Section
\ref{sec-extensions}, we extend results of \cite{loday-pira} in the context of LeibDer pairs. We prove that isomorphism classes of
central extensions of a LeibDer pair by a trivial LeibDer pair are classified by the second cohomology
group of the LeibDer pair with coefficients in the trivial representation (cf. Theorem \ref{final-theorem-central}). Next we
discuss the extension of a pair of derivations in a central extension of Leibniz algebras. Given a
central extension $0 \rightarrow \mathfrak{a} \xrightarrow{i} \mathfrak{h} \xrightarrow{p} \mathfrak{g} \rightarrow 0$ of Leibniz algebras and a pair of derivations $(\phi_\mathfrak{g}, \phi_\mathfrak{a}) \in \mathrm{Der} (\mathfrak{g}) \times \mathrm{Der} (\mathfrak{a})$, we associate a second cohomology class in the cohomology of the Leibniz algebra $\mathfrak{g}$ with coefficients in the trivial representation $\mathfrak{a}$. When this cohomology class is null, the pair $(\phi_\mathfrak{g}, \phi_\mathfrak{a}) \in \mathrm{Der} (\mathfrak{g}) \times \mathrm{Der} (\mathfrak{a})$ of derivations extends to a derivation $\phi_\mathfrak{h} \in \mathrm{Der}(\mathfrak{h})$ which makes the above sequence into an exact sequence of LeibDer pairs (cf. Theorem \ref{final-theorem-pair}). In subsection \ref{subsec-abelian-extension}, we consider abelian extensions of a LeibDer pair by any arbitrary representation. In Theorem \ref{final-theorem-abelian} we prove that isomorphism classes of abelian extensions are classified by the second cohomology group of the LeibDer pair.

\medskip

The classical deformation theory of Gerstenhaber \cite{gers} has been extended to Leibniz algebras in \cite{bala-leib}.
In Section \ref{sec-deformation}, we generalize this deformation theory to LeibDer pairs. Our results in this section are
analogous to the standard ones. The vanishing of the second cohomology implies that the LeibDer
pair is rigid, i.e, any deformation is equivalent to the undeformed one (cf. Theorem \ref{two-rigidity}). Given a finite order deformation, we associate a third cohomology class in the cohomology of the LeibDer
pair, called the obstruction class (cf. Proposition \ref{ob-three-cocycle}). When this class is null, the given deformation
extends to a deformation of next order (cf. Theorem \ref{ob-three-cocycle-thm}).

\medskip

The notion of sh Leibniz algebras was introduced in \cite{ammar-poncin}. In \cite{sheng-liu} the authors consider $2$-term sh Leibniz algebras and relate them with categorified Leibniz
algebras. In Section \ref{sec-homo-der}, we introduce homotopy derivations on $2$-term sh Leibniz algebras. Homotopy derivations on skeletal sh Leibniz algebras are characterized by third cocycles of LeibDer pairs
(cf. Proposition \ref{skeletal-prop}). We introduce crossed modules of LeibDer pairs and prove that strict homotopy
derivations on strict sh Leibniz algebras are characterized by crossed modules of LeibDer pairs (cf Theorem \ref{thm-crossed-leibder}).

\medskip

The notion of Lie $2$-algebras or categorified Lie algebras was first introduced by Baez and Crans \cite{baez-crans}.
They showed that the category of $2$-term $L_\infty$-algebras is equivalent to the category of Lie $2$-algebras. This result has been extended to various other algebraic structures including associative algebras, Leibniz algebras \cite{das-mandal, sheng-liu}. In Section \ref{sec-2-der}, we introduce LeibDer $2$-pairs. They are categorification of LeibDer pairs. Finally, we prove that the category of LeibDer $2$-pairs and the category of $2$-term sh Leibniz algebras with homotopy derivations are equivalent (cf. Theorem \ref{thm-equiv-cat}).

\medskip

Throughout the paper, all vector spaces and linear maps are over a field $\mathbb{K}$ of characteristic zero unless otherwise stated.

\section{Leibniz algebras and their cohomology}
In this section, we recall Leibniz algebras and their cohomology. Our main references are \cite{loday-pira},\cite{bala-leib}.
\begin{defn}
A Leibniz algebra $(\mathfrak{g}, [ ~,~ ])$ consists of a vector space $\mathfrak{g}$ and a bilinear map $[~ , ~ ] :
\mathfrak{g} \times \mathfrak{g} \rightarrow \mathfrak{g}$ satisfying the following identity (Leibniz identity)
\begin{align}\label{leib-defn-id}
 [[x, y], z] = [[x, z], y] + [x, [y, z]]  , ~~~ \text{ for all } x, y, z \in \mathfrak{g}.
\end{align}
\end{defn}

Such Leibniz algebras are called right Leibniz algebras as the identity (\ref{leib-defn-id}) is equivalent to the fact that the maps $[~, z] : \mathfrak{g} \rightarrow \mathfrak{g}$ by fixing right coordinate are derivations for the bracket on $\mathfrak{g}$. Thus, one may also define left Leibniz algebras. In this paper, by a Leibniz algebra, we shall always mean a right Leibniz algebra. However, all the results of the present paper can be easily generalize to left Leibniz algebras by suitable modifications.

\begin{defn}\label{defn-leib-repn}
 Let $(\mathfrak{g}, [~, ~])$ be a Leibniz algebra. A representation of it consists of a vector space
$M$ together with maps (called left and right actions)
\begin{align*}
[~,~] : \mathfrak{g} \times M \rightarrow M \qquad  [~,~] : M \times \mathfrak{g} \rightarrow M
\end{align*}
satisfying the following set of identities
\begin{itemize}
\item[(MLL)] $ [[m, x], y] = [[m, y], x] + [m, [x, y]],$
\item[(LML)] $ [[x, m], y] =  [[x, y], m] + [x, [m, y]],$
\item[(LLM)] $ [[x, y], m] = [[x, m], y] + [x, [y, m]],$
\end{itemize}
for any $x, y \in \mathfrak{g}$ and $m \in  M$.
\end{defn}
Note that $\mathfrak{g}$ is a representation of itself with left and right actions are given by the bracket on $\mathfrak{g}$.

Let $(\mathfrak{g}, [~, ~])$ be a Leibniz algebra and $M$ be a representation of it. Define the $n$-th cochain group
\begin{align*}
C^n (\mathfrak{g}, M) := \mathrm{Hom} ( \mathfrak{g}^{\otimes n }, M), ~~~ \text{ for } n \geq 0
\end{align*}
and a map $\delta_L : C^n (\mathfrak{g}, M) \rightarrow C^{n+1} (\mathfrak{g}, M)$ by
\begin{align*}
(\delta_L f) (x_1, \ldots, x_{n+1}) = [x_1, f(x_2, \ldots, x_{n+1})] + \sum_{i=2}^{n+1} (-1)^i [ f( x_1, \ldots, \hat{x_i}, \ldots, x_{n+1}), x_i ] \\
+ \sum_{i < j} (-1)^{j+1} f(x_1, \ldots, x_{i-1}, [x_i, x_j], x_{i+1}, \ldots, \hat{x_j}, \ldots, x_{n+1}),
\end{align*}
for $ x_1, \ldots,  x_{n+1} \in \mathfrak{g}$. Then one has $(\delta_L)^2 = 0$.
Therefore, $(C^\ast (\mathfrak{g}, M), \delta_L)$ is a cochain
complex. The cohomology of the Leibniz algebra $\mathfrak{g}$ with coefficients in $M$ is denoted by
\begin{align*}
H^n (\mathfrak{g}, M) := \frac{\text{ker} \{ \delta_L : C^n (\mathfrak{g}, M) \rightarrow C^{n+1} (\mathfrak{g}, M) \}  }{\text{Im} \{ \delta_L : C^{n-1} (\mathfrak{g}, M) \rightarrow C^n (\mathfrak{g}, M) \} }.
\end{align*}

\medskip

When we consider the cohomology of a Leibniz algebra $\mathfrak{g}$ with coefficients in itself, the graded space $C^\ast (\mathfrak{g}, \mathfrak{g}) = \oplus_n C^n(\mathfrak{g}, \mathfrak{g})$ of cochain groups carries a degree $-1$ graded Lie bracket given by $\llbracket f, g \rrbracket = f \bullet g -  (-1)^{(m-1)(n-1)} g \bullet f,$ for $f \in C^m(\mathfrak{g} , \mathfrak{g} ), ~ g \in  C^n(\mathfrak{g} , \mathfrak{g} ),$ where
\begin{align*}
&(f \bullet g) (x_1, \ldots, x_{m+n-1}) = \\
& \sum_{i=1}^m (-1)^{(i-1)(n-1)} \sum_{  \mathrm{Sh} (n-1, m-i)} \mathrm{sgn} (\sigma) f (x_1, \ldots, x_{i-1}, g (x_i, x_{\sigma (i+1)}, \ldots, x_{\sigma (i+n-1} ), x_{\sigma (i+n)}, \ldots, x_{\sigma (m+n-1} ),
\end{align*}
for $x_1, \ldots, x_{m+n-1} \in \mathfrak{g}$. If we denote the Leibniz bracket on $\mathfrak{g}$ by the bilinear map $\mu : \mathfrak{g}^{\otimes 2} \rightarrow \mathfrak{g}$ (i.e, $\mu (x,y) = [x, y]$, for $x, y \in \mathfrak{g}$) then the Leibniz identity for the bracket is equivalent to $\llbracket \mu, \mu \rrbracket = 0$, where $\mu$ is considered as an element in $C^2 (\mathfrak{g}, \mathfrak{g} ).$ With this notation, the differential (with coefficients in $\mathfrak{g}$) is given by
\begin{align*}
\delta_L f = (-1)^{n-1} \llbracket \mu, f \rrbracket,  ~~~ \text{ for } f \in C^n(\mathfrak{g}, \mathfrak{g}).
\end{align*}
This implies that the graded space of cohomology $H^\ast (\mathfrak{g}, \mathfrak{g})$ carries a degree $-1$ graded Lie bracket.

\section{LeibDer pairs}
In this section, we consider LeibDer pairs. We define representations and cohomology of a LeibDer pair. Finally the cohomology of a LeibDer pair with coefficients in itself carries a degree $-1$ graded Lie bracket.

Let $\mathfrak{g}$ be a Leibniz algebra. A linear map $\phi_\mathfrak{g} : \mathfrak{g} \rightarrow \mathfrak{g}$ is a derivation on $\mathfrak{g}$ if it satisfies
\begin{align*}
\phi_\mathfrak{g} [x, y] = [ \phi_\mathfrak{g} (x), y] + [ x, \phi_\mathfrak{g} (y)], ~~~ \text{ for } x, y \in \mathfrak{g}.
\end{align*}
Note that derivations are $1$-cocycles in the cohomology complex of $\mathfrak{g}$ with coefficients in $\mathfrak{g}$. We call a pair $(\mathfrak{g}, \phi_\mathfrak{g})$ of a Leibniz algebra $\mathfrak{g}$ and a derivation $\phi_\mathfrak{g}$, a LeibDer pair. When the Leibniz algebra bracket is skew-symmetric, one gets the notion of LieDer pair \cite{tang}. Thus, LeibDer pairs are non-skewsymmetric analog of LieDer pairs.

\begin{exam}

\end{exam}

$\bullet$ Let $(\mathfrak{g}, [~, ~])$ is a Leibniz algebra. Then for any $x \in \mathfrak{g}$, the linear map $\phi_x := [~, x] : \mathfrak{g} \rightarrow \mathfrak{g}$ is a derivation on $\mathfrak{g}$. Hence, $(\mathfrak{g}, \phi_x)$ is a LeibDer pair.

$\bullet$ Any associative dialgebra gives rise to a Leibniz algebra in the same way an associative algebra gives a Lie algebra. An associative dialgebra is a vector space $D$ together with two bilinear maps $\dashv , \vdash : D \times D \rightarrow D$ satisfying five associative style identities \cite{loday-dialgebra}. A linear map $d: D \rightarrow D$ is a derivation for the associative dialgebra if $d$ is a derivation for both the products $\dashv$ and $\vdash$.

If $(D, \dashv, \vdash)$ is an associative dialgebra, then $D$ equipped with the bracket
\begin{align*}
[x, y]:= x \dashv y - y \vdash x
\end{align*}
is a Leibniz algebra. Further, if  $d$ is a derivation for the associative dialgebra, then $d$ is also a derivation for the induced Leibniz algebra structure on $D$.

$\bullet$ Let $(L, [~, \ldots, ~])$ be a $n$-Leibniz algebra, i.e, $[~, \ldots, ~] : L^{\times n} \rightarrow L$ is a multilinear map satisfying
\begin{align*}
[[x_1, \ldots, x_n], y_1, \ldots, y_{n-1}] = \sum_{i=1}^n [ x_1, \ldots, x_{i-1}, [ x_i, y_1, \ldots, y_{n-1}], x_{i+1}, \ldots, x_n],
\end{align*}
for $x_1, \ldots, x_n, y_1, \ldots, y_{n-1} \in L$ \cite{casas-loday-pira}. A derivation on the $n$-Leibniz algebra $L$ is given by a linear map $d : L \rightarrow L$ that satisfies
$d [x_1, \ldots, x_n] = \sum_{i=1}^n [x_1, \ldots, dx_i, \ldots, x_n].$
Note that a $n$-Leibniz algebra $(L, [~, \ldots, ~])$ induces a Leibniz algebra structure on $L^{\otimes n-1}$ with bracket
\begin{align*}
[ x_1 \otimes \cdots \otimes x_{n-1}, y_1 \otimes \cdots \otimes y_{n-1}] = \sum_{i=1}^{n-1} x_1 \otimes \cdots \otimes [x_i, y_1, \ldots, y_{n-1}] \otimes \cdots \otimes x_{n-1}.
\end{align*}
If $d : L \rightarrow L$ is a derivation for the $n$-Leibniz algebra $L$, then $d$ induces a map $\overline{d} : L^{\otimes n-1} \rightarrow L^{\otimes n-1}$ by
\begin{align*}
\overline{d} ( x_1 \otimes \cdots \otimes x_{n-1}) = \sum_{i=1}^{n-1} x_1 \otimes \cdots \otimes dx_i \otimes \cdots \otimes x_{n-1}.
\end{align*}
It can be checked that $\overline{d}$ is a derivation for the Leibniz bracket on $L^{\otimes n-1}$. In other words, $(L^{\otimes n-1}, \overline{d})$ is a LeibDer pair.

\begin{remark} Any Leibniz algebra associates a Lie algebra via skew-symmetrization. Let $(\mathfrak{g}, [~, ~])$ be a Leibniz algebra. Then the associated Lie algebra $\mathfrak{g}_{\mathrm{Lie}}$ is the quotient of $\mathfrak{g}$ by the subspace $S$ generated by elements of the form $[x, x]$, for $x \in \mathfrak{g}$. The Lie bracket on $\mathfrak{g}_{\mathrm{Lie}}$ is the one induced from the Leibniz bracket on $\mathfrak{g}$. If $(\mathfrak{g}, \phi_\mathfrak{g})$ is a LeibDer pair, then we have
\begin{align*}
\phi_\mathfrak{g} [x, x] &= [ \phi_\mathfrak{g} (x), x] + [ x, \phi_\mathfrak{g}(x)] \\
~&= [ x + \phi_\mathfrak{g} (x) , x+ \phi_\mathfrak{g} (x)] - [ x, x] - [ \phi_\mathfrak{g} (x), \phi_\mathfrak{g} (x) ] \in S.
\end{align*}
Therefore, $\phi_\mathfrak{g}$ induces a map $\overline{\phi_\mathfrak{g}} : \mathfrak{g}_{\mathrm{Lie}} \rightarrow \mathfrak{g}_{\mathrm{Lie}}$. Since $\phi_\mathfrak{g}$ is a derivation for the Leibniz bracket, $\overline{ \phi_\mathfrak{g}}$ is a Lie algebra derivation. In other words, $(\mathfrak{g}_{\mathrm{Lie}}, \overline{\phi_\mathfrak{g}} )$ is a LieDer pair.
\end{remark}

\begin{defn}
Let $(\mathfrak{g}, \phi_\mathfrak{g})$ and $(\mathfrak{h}, \phi_\mathfrak{h})$ be two LeibDer pairs. A LeibDer pair morphism between them is a Leibniz algebra morphism $f : \mathfrak{g} \rightarrow \mathfrak{h}$ satisfying $\phi_\mathfrak{h} \circ f = f \circ \phi_\mathfrak{g}$. It is called LeibDer pair isomorphism if $f$ is an isomorphism. 
\end{defn}
LeibDer pairs and morphisms between them forms a category, denoted by {\bf LeibDer}.

Let $V$ be a vector space. Consider the tensor module $\overline{T}(V) = V \oplus V^{\otimes 2} \oplus V^{\otimes 3} \oplus \cdots $ with the bracket inductively defined by
\begin{align*}
&[x,v ] = x \otimes v,\\
&[x, y \otimes v] = [x, y] \otimes v - [ x \otimes v, y], ~~~ \text{ for } x , y \in \overline{T}(V), v \in V.
\end{align*}
Then $\overline{T}(V)$ with the above bracket is a free Leibniz algebra over $V$ \cite{loday-pira}. Any linear map $d : V \rightarrow V$ induces a linear map $\overline{d} : \overline{T}(V)  \rightarrow \overline{T}(V) $ by
\begin{align}\label{tensor-deri}
\overline{d} (v_1 \otimes \cdots \otimes v_n ) = \sum_{i=1}^n v_1 \otimes \cdots \otimes dv_i \otimes \cdots \otimes v_n.
\end{align}
Then $\overline{d}$ is a derivation for the Leibniz bracket on $\overline{T}(V)$. To check that, we first observe that
\begin{align*}
\overline{d} [ v_1 \otimes \cdots \otimes v_n, v ] =~& \overline{d} (v_1 \otimes \cdots \otimes v_n \otimes v ) \\
=~& \sum_{i=1}^n  v_1 \otimes \cdots \otimes dv_i \otimes \cdots \otimes v_n \otimes v ~+~ v_1 \otimes \cdots \otimes v_n \otimes dv\\
=~& [ \overline{d} (v_1 \otimes \cdots \otimes v_n), v] + [ v_1 \otimes \cdots \otimes v_n, \overline{d} (v)].
\end{align*}
Next suppose that $\overline{d} [x, v_1 \otimes \cdots \otimes v_k] = [ \overline{d} (x), v_1 \otimes \cdots \otimes v_k] + [ x, \overline{d} (v_1 \otimes \cdots \otimes v_k)]$, for $x \in \overline{T}(V)$, $v_1, \ldots, v_k \in V$ and $ k \geq 1$. Then we have
\begin{align*}
&\overline{d} [x, v_1 \otimes \cdots \otimes v_{k+1}] \\
&= \overline{d} ([x, v_1 \otimes \cdots \otimes v_k] \otimes v_{k+1} - [ x \otimes v_{k+1}, v_1 \otimes \cdots \otimes v_k]) \\
& = \overline{d} [x, v_1 \otimes \cdots \otimes v_k] \otimes v_{k+1} + [x, v_1 \otimes \cdots \otimes v_k] \otimes d v_{k+1} - [ \overline{d} (x \otimes v_{k+1}),  v_1 \otimes \cdots \otimes v_k] \\
& \quad - [ x \otimes v_{k+1}, \overline{d} (v_1 \otimes \cdots \otimes v_k)]\\
 &= [ \overline{d} (x), v_1 \otimes \cdots \otimes v_k] \otimes v_{k+1} + [ x, \overline{d} (v_1 \otimes \cdots \otimes v_k)] \otimes v_{k+1} + [ x, v_1 \otimes \cdots \otimes v_k] \otimes dv_{k+1}\\
& \quad - [ \overline{d} (x) \otimes v_{k+1}, v_1 \otimes \cdots \otimes v_k] - [ x \otimes dv_{k+1} , v_1 \otimes \cdots \otimes v_k] - [ x \otimes v_{k+1}, \overline{d} (v_1 \otimes \cdots \otimes v_k)]\\
& = [ \overline{d} (x), v_1 \otimes \cdots \otimes v_{k+1}]  + [x,  \overline{d} (v_1 \otimes \cdots \otimes v_k) \otimes v_{k+1}] + [ x, (v_1 \otimes \cdots \otimes v_k)\otimes dv_{k+1}] \\
& = [ \overline{d} (x), v_1 \otimes \cdots \otimes v_{k+1}] + [x, \overline{d}(v_1 \otimes \cdots \otimes v_{k+1})]. 
\end{align*}
Hence $(\overline{T}(V), \overline{d})$ is a LeibDer pair. In the next, we show that this LeibDer pair is free in the following sense.

Let $(V, d)$ be a pair of vector space $V$ and a linear map $d: V \rightarrow V$. The free LeibDer pair over $(V, d)$ is a LeibDer pair $(\mathcal{L}(V), \phi_{\mathcal{L}(V)})$ equipped with a linear map $i : V \rightarrow \mathcal{L}(V)$ satisfying $\phi_{\mathcal{L}(V)} \circ i = i \circ d$ and the following universal condition holds: for any LeibDer pair $(\mathfrak{g}, \phi_\mathfrak{g})$ and a linear map $f : V \rightarrow \mathfrak{g}$ satisfying $\phi_\mathfrak{g} \circ f = f \circ d$, there exists an unique LeibDer pair morphism $\widetilde{f} : (\mathcal{L}(V), \phi_{\mathcal{L}(V)}) \rightarrow (\mathfrak{g}, \phi_\mathfrak{g})$ such that $\widetilde{f} \circ i = f$.

\begin{prop}
The LeibDer pair $(\overline{T}(V), \overline{d})$ is free over $(V, d)$.
\end{prop}

\begin{defn}
Let $(\mathfrak{g}, \phi_\mathfrak{g})$ be a LeibDer pair. A representation of it is given by a pair $(M, \phi_M)$ in which $M$ is a representation of $\mathfrak{g}$ (see Definition \ref{defn-leib-repn}) and $\phi_M : M \rightarrow M$ is a linear map satisfying
\begin{align}
&\phi_M [x,m] = [ \phi_\mathfrak{g} (x), m] + [ x, \phi_M (m)], \label{rep-id-1}\\
&\phi_M [m,x] = [ \phi_M (m), x] + [ m, \phi_\mathfrak{g} (x)]. \label{rep-id-2}
\end{align}
\end{defn}

It is known that a representation of a Leibniz algebra gives rise to a semi-direct product \cite{loday-pira}. We extend this in the context of LeibDer pairs.

\begin{prop}
Let $(\mathfrak{g}, \phi_\mathfrak{g})$ be a LeibDer pair and $(M, \phi_M)$ be a representation of it. Then $( \mathfrak{g} \oplus M, \phi_\mathfrak{g} \oplus \phi_M)$ is a LeibDer pair where the Leibniz algebra bracket on $\mathfrak{g} \oplus M$ is given by the semi-direct product
\begin{align*}
[(x,m), (y,n)] = ([x,y], [x,n]+ [m, y]).
\end{align*}
\end{prop}

\begin{proof}
We only need to show that $\phi_\mathfrak{g} \oplus \phi_M :  \mathfrak{g} \oplus M \rightarrow  \mathfrak{g} \oplus M$ is a derivation for the Leibniz algebra  $\mathfrak{g} \oplus M$. We have 
\begin{align*}
(\phi_\mathfrak{g} \oplus \phi_M )  [(x,m), (y,n)] =~& (\phi_\mathfrak{g} [x, y],  \phi_M [x, n] + \phi_M [m, y] ) \\
=~& ([\phi_\mathfrak{g} (x), y] , [ \phi_\mathfrak{g} (x), n] + [ \phi_M (m), y] ) + ([x, \phi_\mathfrak{g} (y)] , [x, \phi_M (n)] + [m, \phi_\mathfrak{g} (x)])\\
=~& [(\phi_\mathfrak{g} \oplus \phi_M) (x, m), (y, n) ] + [ (x, m) ,  (\phi_\mathfrak{g} \oplus \phi_M) (y, n)]. 
\end{align*}
Hence the proof.
\end{proof}

\subsection{Universal enveloping AssDer pair}
In \cite{loday-pira} the authors construct the universal enveloping algebra $UL(\mathfrak{g})$ of a Leibniz algebra $\mathfrak{g}$. When the Leibniz algebra $\mathfrak{g}$ is equipped with a derivation, it induces a derivation on $UL(\mathfrak{g})$. We first recall the construction of $UL(\mathfrak{g})$.

Let $\mathfrak{g}^l$ and $\mathfrak{g}^r$ be two copies of the Leibniz algebra $\mathfrak{g}$. For any $x \in \mathfrak{g}$, we denote the corresponding element in $\mathfrak{g}^l$ and $\mathfrak{g}^r$ by $l_x$ and $r_x$, respectively. Then the universal enveloping algebra $UL(\mathfrak{g})$ is the quotient of the tensor algebra $T(\mathfrak{g}^l \oplus \mathfrak{g}^r)$ by the two sided ideal $I$ generated by elements of the form 
\begin{align*}
r_x \otimes r_y - r_y \otimes r_x - r_{[x,y]}, \qquad l_x \otimes r_y - r_y \otimes l_x - l_{[x,y]} \quad \text{ and } \quad (r_y + l_y) \otimes l_x, ~~~ \text{ for } x, y \in \mathfrak{g}. 
\end{align*}

Let $(\mathfrak{g}, \phi_\mathfrak{g})$ be a LeibDer pair. Then $\phi_\mathfrak{g}$ induces a linear map (denoted by the same notation) $\phi_\mathfrak{g} : \mathfrak{g}^l \oplus \mathfrak{g}^r \rightarrow \mathfrak{g}^l \oplus \mathfrak{g}^r$ by $\phi_\mathfrak{g} (l_x, r_y) = (l_{\phi_\mathfrak{g} (x)}, r_{\phi_\mathfrak{g} (y)}).$ This linear map on $\mathfrak{g}^l \oplus \mathfrak{g}^r$ induces a derivation $\overline{\mathfrak{g}}$ on the tensor algebra $T(\mathfrak{g}^l \oplus \mathfrak{g}^r)$. See Equation (\ref{tensor-deri}). Then $\overline{\phi_\mathfrak{g}} (I) \subset I$ as
\begin{align*}
&\overline{\phi_\mathfrak{g}}  (r_x \otimes r_y - r_y \otimes r_x - r_{[x,y]} ) \\
&= r_{\phi_\mathfrak{g}(x)} \otimes r_y + r_x \otimes r_{\phi_\mathfrak{g} (y)} - r_{\phi_\mathfrak{g} (y)} \otimes r_x - r_y \otimes  r_{\phi_\mathfrak{g}(x)} - r_{\phi_\mathfrak{g} [x,y]} \\
&= r_{\phi_\mathfrak{g}(x)} \otimes r_y - r_y \otimes  r_{\phi_\mathfrak{g}(x)} - r_{[\phi_\mathfrak{g} (x), y]} + r_x \otimes r_{\phi_\mathfrak{g} (y)} - r_{\phi_\mathfrak{g} (y)} \otimes r_x - r_{[x, \phi_\mathfrak{g} (y)]} \in I,
\end{align*}
\begin{align*}
&\overline{\phi_\mathfrak{g}}  (l_x \otimes r_y - r_y \otimes l_x - l_{[x,y]} )\\
& = l_{\phi_\mathfrak{g}(x)} \otimes r_y - r_y \otimes  l_{\phi_\mathfrak{g}(x)} - l_{[\phi_\mathfrak{g} (x), y]} + l_x \otimes r_{\phi_\mathfrak{g} (y)} - r_{\phi_\mathfrak{g} (y)} \otimes l_x - l_{[x, \phi_\mathfrak{g} (y)]} \in I
\end{align*}
and
\begin{align*}
\overline{\phi_\mathfrak{g}} ( (r_y + l_y) \otimes l_x ) = (r_{\phi_\mathfrak{g}(y)} + l_{\phi_\mathfrak{g} (y)} ) \otimes l_x + (r_y + l_y) \otimes l_{\phi_\mathfrak{g} (x)} \in I.
\end{align*}
Hence $\overline{\phi_\mathfrak{g}} $ induces a derivation on the universal enveloping associative algebra $UL(\mathfrak{g}) = T (\mathfrak{g}^l \oplus \mathfrak{g}^r) / I$. In other words, $(UL(\mathfrak{g}), \overline{\phi_\mathfrak{g}})$ is an AssDer pair.

In \cite{loday-pira} the authors showed that representations of a Leibniz algebra $\mathfrak{g}$ is equivalent to right modules over $UL(\mathfrak{g})$. More precisely, let $M$ be a representation of a Leibniz algebra $\mathfrak{g}$. Then the right action of $UL (\mathfrak{g})$ on $M$ is given by
\begin{align*}
m \cdot l_x = [x, m] \qquad \text{ and } \qquad m \cdot r_x = [m, x].
\end{align*}
We extend this situation endowed with derivations. We first recall the following definition from \cite{das-mandal}.

Let $(A, \phi_A)$ be an AssDer pair. A right module over it consists of a pair $(M, \phi_M)$ in which $M$ is a right $A$-module and $\phi_M: M \rightarrow M$ is a linear map satisfying
\begin{align*}
\phi_M (m \cdot a) = \phi_M (m) \cdot a + m \cdot \phi_A (a), ~~~ \text{ for } a \in A, m \in M.
\end{align*}

\begin{prop}
The category of representations of a LeibDer pair $(\mathfrak{g}, \phi_\mathfrak{g})$ is equivalent to the category of right modules over the AssDer pair $(UL(\mathfrak{g}), \overline{\phi_\mathfrak{g}}).$
\end{prop}

\begin{proof}
Let $(M, \phi_M)$ be a representation of the LeibDer pair $(\mathfrak{g}, \phi_\mathfrak{g})$. Then $M$ is already a right $UL(\mathfrak{g})$-module. Moreover, the condition (\ref{rep-id-1}) is equivalent to $\phi_M (m \cdot l_x) = \phi_M (m) \cdot l_x + m \cdot \overline{\phi_\mathfrak{g}} (l_x)$. Similarly, the condition (\ref{rep-id-2}) is equivalent to $\phi_M (m \cdot r_x) = \phi_M (m) \cdot r_x + m \cdot \overline{\phi_\mathfrak{g}} (r_x)$. Hence the proof.
\end{proof}

\subsection{Cohomology}

Here we define cohomology of a LeibDer pair with coefficients in a representation. Let $(\mathfrak{g}, \phi_\mathfrak{g})$ be a LeibDer pair and $(M, \phi_M)$ be a representation of it. We define the cochain groups by $C^0_{\mathrm{LeibDer}} (\mathfrak{g}, M) := 0$, $C^1_{\mathrm{LeibDer}} (\mathfrak{g}, M) := \mathrm{Hom} (\mathfrak{g}, M)$ and
\begin{align*}
C^n_{\mathrm{LeibDer}} (\mathfrak{g}, M) := \mathrm{Hom} (\mathfrak{g}^{\otimes n} , M ) \times  \mathrm{Hom} (\mathfrak{g}^{\otimes n -1 } , M ) , ~~~ \text{ for } n \geq 2.
\end{align*}
Before we define the coboundary operator, we define a new map $\delta : \mathrm{Hom} (\mathfrak{g}^{\otimes n} , M ) \rightarrow \mathrm{Hom} (\mathfrak{g}^{\otimes n} , M )$ by
\begin{align*}
\delta f = \sum_{i=1}^n f \circ (\mathrm{id} \otimes \cdots \otimes \phi_\mathfrak{g} \otimes \cdots \otimes \mathrm{id}) - \phi_M \circ f.
\end{align*}
Then we have the following.
\begin{lemma}
The map $\delta$ commute with $\delta_L$, i.e, $\delta_L \circ \delta = \delta \circ \delta_L$.
\end{lemma}
Finally, we define the coboundary operator $\partial : C^n_{\mathrm{LeibDer}} (\mathfrak{g}, M) \rightarrow C^{n+1}_{\mathrm{LeibDer}} (\mathfrak{g}, M)$ as
\begin{align*}
 \begin{cases}
 \partial f = (\delta_L f , - \delta f ), ~~~ \text{ for } f \in C^1_{\mathrm{LeibDer}} (\mathfrak{g}, M), \\
 \partial ( f_n, \overline{f}_n) = (\delta_L f_n, \delta_L \overline{f}_n + (-1)^n \delta f_n), ~~~ \text{ for } ( f_n, \overline{f}_n) \in C^n_{\mathrm{LeibDer}} (\mathfrak{g}, M).
 \end{cases}
\end{align*}
\begin{prop}
The map $\partial$ satisfies $\partial^2 = 0$.
\end{prop}
\begin{proof}
For $f \in C^1_{\mathrm{LeibDer}} (\mathfrak{g}, M)$, we have
\begin{align*}
\partial^2 f = \partial ( \delta_{L} f, - \delta f) = ( \delta_{L}^2 f, ~ - \delta_{L} \delta f + \delta \delta_{L} f ) = 0.
\end{align*}
Similarly, for $(f_n, \overline{f}_n ) \in C^n_{\mathrm{LeibDer}} (\mathfrak{g}, M)$, we have
\begin{align*}
\partial^2 (f_n, \overline{f}_n ) =~& \partial ( \delta_{L} f_n, ~ \delta_{L} \overline{f}_n + (-1)^n \delta f_n ) \\
=~& ( \delta_{L}^2 f_n, ~ \delta_{L}^2 \overline{f}_n + (-1)^n \delta_{L} \delta f_n + (-1)^{n+1} \delta \delta_{L} f_n ) = 0.
 \end{align*}
 Hence the proof.
\end{proof}
Thus, it follows from the above proposition that $(C^\ast_{\mathrm{LeibDer}}(\mathfrak{g}, M), \partial )$ is a cochain complex. The corresponding cohomology groups are denoted by $H^n_{\mathrm{LeibDer}}(\mathfrak{g}, M), ~ n \geq 0$.

In the next, we show that the cohomology of a LeibDer pair $(\mathfrak{g} , \phi_\mathfrak{g})$ with coefficients in itself carries a degree $-1$ graded Lie bracket.

\begin{prop}
The bracket $\llbracket ~, ~ \rrbracket^{\thicksim} : C^m_{\mathrm{LeibDer}} (\mathfrak{g}, \mathfrak{g}) \times C^n_{\mathrm{LeibDer}} (\mathfrak{g}, \mathfrak{g}) \rightarrow C^{m+n-1}_{\mathrm{LeibDer}} (\mathfrak{g}, \mathfrak{g})$ given by
\begin{align*}
\llbracket (f, \overline{f}), (g, \overline{g}) \rrbracket^{\thicksim}  := ( \llbracket f, g \rrbracket,  (-1)^{m+1} \llbracket f, \overline{g} \rrbracket + \llbracket \overline{f}, g \rrbracket )
\end{align*}
defines a degree $-1$ graded Lie bracket on the graded space $C^\ast_{\mathrm{LeibDer}} (\mathfrak{g}, \mathrm{g})$.
\end{prop}

The proof of this result follows as $\llbracket ~, ~ \rrbracket$ is a degree $-1$ graded Lie bracket on $C^\ast (\mathfrak{g}, \mathfrak{g})$. See \cite{das-mandal} for similar result for AssDer pairs. 
With this new bracket on  $C^\ast_{\mathrm{LeibDer}} (\mathfrak{g}, \mathrm{g})$, we have $\llbracket (\mu, \phi_\mathfrak{g}), (\mu, \phi_\mathfrak{g}) \rrbracket^{\thicksim} = 0$, where $ (\mu, \phi_\mathfrak{g}) \in \mathrm{Hom} (\mathfrak{g}^{\otimes 2} , \mathfrak{g}) \times \mathrm{Hom} (\mathfrak{g} , \mathfrak{g})$ considered as an element in $C^2_{\mathrm{LeibDer}} (\mathfrak{g}, \mathfrak{g})$. Thus $(\mu , \phi_\mathfrak{g}) \in C^2_{\mathrm{LeibDer}} (\mathfrak{g}, \mathfrak{g})$ is a Maurer-Cartan element in the graded Lie algebra $(C^{\ast + 1}_{\mathrm{LeibDer}} (\mathfrak{g}, \mathfrak{g}) , \llbracket ~, ~ \rrbracket^{\thicksim} )$. Moreover, the differential $\partial$ is induced by the Maurer-Cartan element as
\begin{align*}
\partial (f, \overline{f} ) = (-1)^{n-1} \llbracket (\mu, \phi_\mathfrak{g}), (f, \overline{f} )  \rrbracket^{\thicksim}, ~~~ \text{ for } (f, \overline{f} )  \in C^n_{\mathrm{LeibDer}} (\mathfrak{g}, \mathfrak{g}).
\end{align*}
This in particular implies that the graded space of cohomology $H^{\ast +1}_{\mathrm{LeibDer}} (\mathfrak{g}, \mathfrak{g})$ carries a graded Lie algebra structure.

\section{Extensions of LeibDer pairs}\label{sec-extensions}

\subsection{Central extensions}
Central extensions of Leibniz algebras was defined in \cite{loday-pira}. In this subsection, we extend this to LeibDer pairs. We show that isomorphism classes of central extensions are classified by the second cohomology group of LeibDer pair with coefficients in the trivial representation.

Let $(\mathfrak{g}, \phi_\mathfrak{g})$ be a LeibDer pair and $(\mathfrak{a}, \phi_\mathfrak{a})$ be an abelian LeibDer pair, i.e, the Leibniz bracket of $\mathfrak{a}$ is trivial.

\begin{defn}
A central extension of $(\mathfrak{g}, \phi_\mathfrak{g})$  by the abelian LeibDer pair $(\mathfrak{a}, \phi_\mathfrak{a})$ is an exact sequence of LeibDer pairs
\begin{align}\label{central-ext-eqn}
\xymatrix{
0 \ar[r] & (\mathfrak{a}, \phi_\mathfrak{a}) \ar[r]^i & (\mathfrak{h}, \phi_\mathfrak{h}) \ar[r]^p & (\mathfrak{g}, \phi_\mathfrak{g}) \ar[r] & 0
}
\end{align}
such that $[i(a), h] = 0 = [h, i(a)]$, for all $a \in \mathfrak{a}$ and $h \in \mathfrak{h}$.
\end{defn}

One may identify $\mathfrak{a}$ with the corresponding subalgebra of $\mathfrak{h}$  (via the map $i$). With this identification, we have $\phi_\mathfrak{a} = \phi_\mathfrak{h}|_\mathfrak{a}.$

\begin{defn}
Two central extensions $(\mathfrak{h}, \phi_\mathfrak{h})$ and $(\mathfrak{h}', \phi_{\mathfrak{h}'})$ are said to be isomorphic if there exists a LeibDer pair isomorphism $\eta : (\mathfrak{h}, \phi_\mathfrak{h}) \rightarrow (\mathfrak{h}', \phi_{\mathfrak{h}'})$ making the following diagram commutative
\begin{align*}
\xymatrix{
0 \ar[r] & (\mathfrak{a}, \phi_\mathfrak{a}) \ar@{=}[d] \ar[r]^i & (\mathfrak{h}, \phi_\mathfrak{h}) \ar[d]_\eta \ar[r]^p & (\mathfrak{g}, \phi_\mathfrak{g}) \ar@{=}[d] \ar[r] & 0\\
0 \ar[r] & (\mathfrak{a}, \phi_\mathfrak{a}) \ar[r]_{i'} & (\mathfrak{h}', \phi_{\mathfrak{h}'}) \ar[r]_{p'} & (\mathfrak{g}, \phi_\mathfrak{g}) \ar[r] & 0
}
\end{align*}
\end{defn}

Let (\ref{central-ext-eqn}) be a central extension of LeibDer pair. A section of the map $p$ is given by a linear map $s : \mathfrak{g} \rightarrow \mathfrak{h}$ satisfying $p \circ s = \mathrm{id}_\mathfrak{g}$. Section of $p$ always exists.

Let $s : \mathfrak{g} \rightarrow \mathfrak{h}$ be any section of $p$. Define two linear maps $\psi : \mathfrak{g} \otimes \mathfrak{g} \rightarrow \mathfrak{a}$ and $\chi : \mathfrak{g} \rightarrow \mathfrak{a}$ by
\begin{align*}
\psi (x, y) := [s(x) , s(y)] - s [x, y] ~~~~ \text{ and } ~~~~ \chi (x) := \phi_\mathfrak{h} (s(x)) - s(\phi_\mathfrak{g} (x)), ~~~ \text{ for } x, y \in \mathfrak{g}.
\end{align*}

Note that the vector space $\mathfrak{h}$ is isomorphic to the direct sum $\mathfrak{g} \oplus \mathfrak{a}$ via the section $s$. Therefore, one may transfer the structures of $\mathfrak{h}$ to that of $\mathfrak{g} \oplus \mathfrak{a}$. More precisely, the induced bracket and the linear map on $\mathfrak{g} \oplus \mathfrak{a}$ are given by
\begin{align*}
[(x,a), (y, b)]_\psi = ([x, y], \psi (x, y)) ~~~ \text{ and } ~~~
\phi_\chi (x, a) = ( \phi_\mathfrak{g} (x), \phi_\mathfrak{a} (a) + \chi (x)).
\end{align*}
With these notations, we have the following.
\begin{prop}\label{2-co-central}
The vector space $\mathfrak{g} \oplus \mathfrak{a}$ with the above bracket is a Leibniz algebra if and only if $\psi$ is a $2$-cocycle in the Leibniz algebra cohomology of $\mathfrak{g}$ with coefficients in the trivial representation $\mathfrak{a}$. Moreover, $\phi_\chi$ is a derivation for the above Leibniz algebra if and only if $\chi$ satisfies $\delta_L (\chi) + \delta \psi = 0$.
\end{prop}

\begin{proof}
The bracket $[~, ~]_\psi$ is a Leibniz bracket if it satisfies
\begin{align*}
[[ (x, a), (y, b)]_\psi, (z, c)]_\psi = [[ (x, a), (z, c)]_\psi , (y, b) ]_\psi + [(x, a), [(y, b), (z, c)]_\psi ]_\psi.
\end{align*}
This is equivalent to
\begin{align*}
\psi ([x, y], z) = \psi ([x, z], y) + \psi (x, [y,z])
\end{align*}
which is same as $\delta_L (\psi) = 0$, where $\delta_L$ is the Leibniz algebra coboundary operator of $\mathfrak{g}$ with coefficients in the trivial representation $\mathfrak{a}$.

The map $\phi_\chi$ is a derivation for the bracket $[~, ~]_\psi$ if 
\begin{align*}
 \phi_\chi [ (x, a), (y, b)]_\psi = [ \phi_\chi (x, a), (y, b)]_\psi  + [ (x, a), \phi_\chi (y, b)]_\psi.
\end{align*}
This condition is equivalent to
\begin{align*}
\phi_\mathfrak{a} (\psi (x, y)) + \chi ([x, y]) = \psi ( \phi_\mathfrak{g} (x), y) + \psi (x, \phi_\mathfrak{g} (y)),
\end{align*}
or, equivalently,  $\big( \delta_L (\chi) + \delta \psi \big)(x, y) = 0$. Hence the proof.
\end{proof}

It follows from the above proposition that the vector space $\mathfrak{g} \oplus \mathfrak{a}$ with the above bracket and linear map forms a LeibDer pair if and only if $(\psi, \chi)$ is a $2$-cocycle of the LeibDer pair $(\mathfrak{g}, \phi_\mathfrak{g})$ with coefficients in the trivial representation $(\mathfrak{a}, \phi_\mathfrak{a}).$

The cohomology class of the $2$-cocycle $(\psi, \chi)$ does not depend on the choice of sections of $p$. Let $s_1$ and $s_2$ be two sections of $p$. Consider the map $u : \mathfrak{g} \rightarrow \mathfrak{a}$ by $u(x) : = s_1 (x) - s_2 (x)$. Then we have
\begin{align*}
\psi_1 (x, y) =~& [s_1(x), s_1 (y)] - s_1 [x, y]\\
=~& [ s_2 (x) + u(x), s_2 (y) + u(y)] - s_2 [x, y] - u[x, y] \\
=~& \psi_2 (x, y) - u[x, y] \qquad (\text{as } u(x), u(y) \in \mathfrak{a})
\end{align*}
and
\begin{align*}
\chi_1 (x) = \phi_\mathfrak{h} ( s_1(x)) - s_1 ( \phi_\mathfrak{g} (x)) =~& \phi_\mathfrak{h} ( s_2 (x) + u(x)) - s_2 ( \phi_\mathfrak{g}(x)) + u (\phi_\mathfrak{g} (x)) \\
=~& \phi_\mathfrak{h} ( s_2 (x)) - s_2 (\phi_\mathfrak{g} (x)) + \phi_\mathfrak{a} ( u(x)) - u (\phi_\mathfrak{g} (x)) \\
=~& \chi_2 (x) + \phi_\mathfrak{a} (u(x)) - u ( \phi_\mathfrak{g} (x)).
\end{align*}
This shows that $(\psi_1, \chi_1) - (\psi_2, \chi_2) = \partial u$. Therefore, the cohomology classes associated to the sections $s_1$ and $s_2$ are same.

\begin{thm}\label{final-theorem-central}
Let $(\mathfrak{g}, \phi_\mathfrak{g})$ be a LeibDer pair and $(\mathfrak{a}, \phi_\mathfrak{a})$ be an abelian LeibDer pair. Then the isomorphism classes of central extensions of $(\mathfrak{g}, \phi_\mathfrak{g})$ by $(\mathfrak{a}, \phi_\mathfrak{a})$ are classified by the second cohomology group $H^2_{\mathrm{LeibDer}} (\mathfrak{g}, \mathfrak{a})$.
\end{thm}

\begin{proof}
Let $(\mathfrak{h}, \phi_\mathfrak{h})$ and $(\mathfrak{h}' , \phi_{\mathfrak{h}'})$ be two isomorphic central extensions. Suppose the isomorphism is given by a map $\eta : (\mathfrak{h}, \phi_\mathfrak{h}) \rightarrow (\mathfrak{h}' , \phi_{\mathfrak{h}'})$. For any section $s$ of the map $p$, we have
\begin{align*}
p' \circ (\eta \circ s) = (p' \circ \eta) \circ s = p \circ s = \mathrm{id}_\mathfrak{g}.
\end{align*}
This shows that $s' := \eta \circ s$ is a section of the map $p'$. Since $\eta$ is a morphism of LeibDer pairs, we have $\eta|_\mathfrak{a} = \mathrm{id}_\mathfrak{a}$. Hence, we get
\begin{align*}
\psi' (x, y) =~& [s'(x), s'(y)] - s' [x, y] = \eta ( [s(x), s(y)] - [x, y]) = \psi (x, y),\\
\chi' (x) =~& \phi_{\mathfrak{h}'} (s'(x)) - s'( \phi_\mathfrak{g} (x)) = \eta ( \phi_\mathfrak{h} (s(x)) - s (\phi_\mathfrak{g} (x)) ) = \chi (x).
\end{align*}
Therefore, isomorphic central extensions produce same $2$-cocycle, hence, corresponds to the same element in $H^2_{\mathrm{LeibDer}} (\mathfrak{g}, \mathfrak{a}).$

To prove the converse part, consider two cohomologous $2$-cocycles $(\psi, \chi)$ and $(\psi', \chi')$. That is, there exists a linear map $v : \mathfrak{g} \rightarrow \mathfrak{a}$ such that $(\psi, \chi) - (\psi', \chi') = \partial v$. Consider the corresponding LeibDer pairs $(\mathfrak{g} \oplus \mathfrak{a}, [~, ~]_\psi, \phi_\chi )$ and $(\mathfrak{g} \oplus \mathfrak{a}, [~, ~]_{\psi'}, \phi_{\chi'} )$ given in Proposition \ref{2-co-central}. These two LeibDer pairs are isomorphic via the map $\eta : \mathfrak{g} \oplus \mathfrak{a} \rightarrow \mathfrak{g} \oplus \mathfrak{a}$ given by $\eta (x, a) = (x, a + v(x)).$ The map $\eta$ is in fact an isomorphism of central extensions. Hence the proof.
\end{proof}

\subsection{Extensions of a pair of derivations}
Let 
\begin{align}\label{cent-ex}
\xymatrix{
0 \ar[r] & \mathfrak{a} \ar[r]^i & \mathfrak{h} \ar[r]^p & \mathfrak{g} \ar[r] & 0
}
\end{align}
be a fixed central extensions of Leibniz algebras. Given a pair of derivations $(\phi_\mathfrak{g}, \phi_\mathfrak{a}) \in \mathrm{Der} (\mathfrak{g}) \times \mathrm{Der} (\mathfrak{a})$, here we study extensions of them to a derivation $\phi_\mathfrak{h} \in \mathrm{Der} (\mathfrak{h})$ which makes
\begin{align}\label{new-cent-ex}
\xymatrix{
0 \ar[r] & (\mathfrak{a}, \phi_\mathfrak{a}) \ar[r]^i & (\mathfrak{h}, \phi_\mathfrak{h}) \ar[r]^p & (\mathfrak{g}, \phi_\mathfrak{g}) \ar[r] & 0
}
\end{align}
into an exact sequence of LeibDer pairs. In other words, $(\mathfrak{h}, \phi_\mathfrak{h})$ is a central extension of the LeibDer pair $(\mathfrak{g}, \phi_\mathfrak{g}) $ by $(\mathfrak{a}, \phi_\mathfrak{a})$. In such a case, the pair  $(\phi_\mathfrak{g}, \phi_\mathfrak{a}) \in \mathrm{Der} (\mathfrak{g}) \times \mathrm{Der} (\mathfrak{a})$ is said to be extensible.

Let $s : \mathfrak{g} \rightarrow \mathfrak{h}$ be a section of (\ref{cent-ex}). We define a map $\psi : \mathfrak{g} \otimes \mathfrak{g} \rightarrow \mathfrak{a}$ by 
\begin{align*}
\psi  (x, y) := [s(x), s(y)] - s[x, y].
\end{align*}
Given a pair of derivations  $(\phi_\mathfrak{g}, \phi_\mathfrak{a}) \in \mathrm{Der} (\mathfrak{g}) \times \mathrm{Der} (\mathfrak{a})$, we define another map
$\mathrm{Ob}^\mathfrak{h}_{(\phi_\mathfrak{g}, \phi_\mathfrak{a}) } : \mathfrak{g}^{\otimes 2} \rightarrow \mathfrak{a}$ by
\begin{align*}
\mathrm{Ob}^\mathfrak{h}_{(\phi_\mathfrak{g}, \phi_\mathfrak{a})}  (x, y) := \phi_\mathfrak{ a} (\psi (x, y)) - \psi ( \phi_\mathfrak{g} (x), y) - \psi ( x, \phi_\mathfrak{g} (y)).
\end{align*}

\begin{prop}
The map $\mathrm{Ob}^\mathfrak{h}_{(\phi_\mathfrak{g}, \phi_\mathfrak{a})} : \mathfrak{g}^{\otimes 2} \rightarrow \mathfrak{a}$ is a $2$-cycle in the cohomology of the Leibniz algebra $\mathfrak{g}$  with coefficients in the trivial representation $\mathfrak{a}$. Moreover, the cohomology class $[\mathrm{Ob}^\mathfrak{h}_{(\phi_\mathfrak{g}, \phi_\mathfrak{a})}] \in H^2 (\mathfrak{g}, \mathfrak{a})$ does not depend on the choice of sections.
\end{prop}

\begin{proof}
First observe that $\psi$ is a $2$-cocycle in the cohomology of the Leibniz algebra $\mathfrak{g}$ with coefficients in the trivial representation $\mathfrak{a}$. Thus, we have
\begin{align*}
&(\delta_L \mathrm{Ob}^\mathfrak{h}_{(\phi_\mathfrak{g}, \phi_\mathfrak{a}) } ) (x, y, z) \\
&= - \mathrm{Ob}^\mathfrak{h}_{(\phi_\mathfrak{g}, \phi_\mathfrak{a}) } ([x, y], z) + \mathrm{Ob}^\mathfrak{h}_{(\phi_\mathfrak{g}, \phi_\mathfrak{a}) } ([x,z], y) + \mathrm{Ob}^\mathfrak{h}_{(\phi_\mathfrak{g}, \phi_\mathfrak{a}) } (x, [y, z]) \\
&= - \cancel{\phi_\mathfrak{a} ( \psi ([x, y], z))} + \psi ( \phi_\mathfrak{g} [x, y], z ) + \psi ( [x, y], \phi_\mathfrak{g} (z))
+ \cancel{\phi_\mathfrak{a} ( \psi ([x, z], y) )} - \psi ( \phi_\mathfrak{g} [x, z] , y ) \\ & \quad - \psi ( [x, z], \phi_\mathfrak{g} (y)) 
+ \cancel{\phi_\mathfrak{a} ( \psi (x, [y,z]) )}  - \psi (\phi_\mathfrak{g} (x), [y,z]) - \psi ( x, \phi_\mathfrak{g} [y, z]) \\
&=  \psi ( [\phi_\mathfrak{g}(x), y], z) + \psi ( [x, \phi_\mathfrak{g} (y)], z) + \psi ( [x, y], \phi_\mathfrak{g} (z)) 
- \psi ( [\phi_\mathfrak{g} (x), z ], y ) - \psi ( [x, \phi_\mathfrak{g} (z)], y) \\ & \quad - \psi ( [x, y], \phi_\mathfrak{g} (z)) - \psi ( \phi_\mathfrak{g} (x), [y, z]) - \psi (x, [\phi_\mathfrak{g}(y), z]) - \psi ( x, [y, \phi_\mathfrak{g} (z) ])\\
&= 0.
\end{align*}
Therefore, $\mathrm{Ob}^\mathfrak{h}_{(\phi_\mathfrak{g}, \phi_\mathfrak{a}) }$ is a $2$-cocycle. To prove the second part, let $s_1$ and $s_2$ be two sections of (\ref{cent-ex}). Consider the map $u : \mathfrak{g} \rightarrow \mathfrak{a}$ given by $u (x) := s_1(x) - s_2 (x).$ Then
\begin{align*}
\psi_1 (x, y) = \psi_2 (x, y) - u [x, y].
\end{align*}
If ${}^{1}\mathrm{Ob}^\mathfrak{h}_{(\phi_\mathfrak{g}, \phi_\mathfrak{a}) }$ and ${}^2\mathrm{Ob}^\mathfrak{h}_{(\phi_\mathfrak{g}, \phi_\mathfrak{a}) }$ denote the two cocycles corresponding to the sections $s_1$ and $s_2$, then
\begin{align*}
& {}^{1}\mathrm{Ob}^\mathfrak{h}_{(\phi_\mathfrak{g}, \phi_\mathfrak{a}) } (x, y) \\
& = \phi_\mathfrak{a} (\psi_1 (x, y)) - \psi_1 ( \phi_\mathfrak{g} (x), y) - \psi_1 ( x, \phi_\mathfrak{g} (y)) \\
&= \phi_\mathfrak{a} ( \psi_2 (x, y)) - \phi_\mathfrak{a} ( u[x, y]) - \psi_2 ( \phi_\mathfrak{g}(x), y) + u ([\phi_\mathfrak{g} (x), y]) - \psi_2 ( x, \phi_\mathfrak{g}(y)) + u ( [x, \phi_\mathfrak{g} (y)]) \\
& = {}^{2}\mathrm{Ob}^\mathfrak{h}_{(\phi_\mathfrak{g}, \phi_\mathfrak{a}) }(x, y) + \delta_L ( \phi_\mathfrak{a} \circ  u - u \circ \phi_\mathfrak{g} ) (x, y).
\end{align*}
This shows that the $2$-cocycles ${}^{1}\mathrm{Ob}^\mathfrak{h}_{(\phi_\mathfrak{g}, \phi_\mathfrak{a}) }$ and ${}^{2}\mathrm{Ob}^\mathfrak{h}_{(\phi_\mathfrak{g}, \phi_\mathfrak{a}) }$ are cohomologous. Hence they correspond to the same cohomology class in $H^2(\mathfrak{g}, \mathfrak{a}).$
\end{proof}

The cohomology class $[\mathrm{Ob}^\mathfrak{h}_{(\phi_\mathfrak{g}, \phi_\mathfrak{a})}]  \in H^2 (\mathfrak{g}, \mathfrak{a}) $ is called the obstruction class to extend the pair of derivations $(\phi_\mathfrak{g}, \phi_\mathfrak{a}).$

\begin{thm}\label{final-theorem-pair}
Let (\ref{cent-ex}) be a central extension of Leibniz algebras. A pair of derivations $(\phi_\mathfrak{g}, \phi_\mathfrak{a}) \in \mathrm{Der} (\mathfrak{g}) \times \mathrm{Der} (\mathfrak{a})$ is extensible if and only if the obstruction class $[\mathrm{Ob}^\mathfrak{h}_{(\phi_\mathfrak{g}, \phi_\mathfrak{a})}]  \in H^2 (\mathfrak{g}, \mathfrak{a}) $ is trivial.
\end{thm}

\begin{proof}
Suppose there exists a derivations $\phi_\mathfrak{h} \in \mathrm{Der} (\mathfrak{h})$ such that (\ref{new-cent-ex}) is an exact sequence of LeibDer pairs. For any $x \in \mathfrak{g},$ we observe that $ p ( \phi_\mathfrak{h} (s(x)) - s (\phi_\mathfrak{g} (x)) ) = 0.$ Hence $ \phi_\mathfrak{h} (s(x)) - s (\phi_\mathfrak{g} (x)) \in \mathrm{ker}(p) = \mathrm{im}(i)$. We define $\lambda : \mathrm{g} \rightarrow \mathfrak{a}$ by
\begin{align*}
\lambda (x)  = \phi_\mathfrak{h} (s(x)) - s (\phi_\mathfrak{g} (x)).
\end{align*}
For any $s (x) + a \in \mathfrak{h}$, we have
\begin{align}\label{phi-h}
\phi_\mathfrak{h} ( s(x) + a ) = \phi_\mathfrak{h} ( s(x)) - s( \phi_\mathfrak{g} (x)) + s (\phi_\mathfrak{g} (x)) + \phi_\mathfrak{a} (a)
= s (\phi_\mathfrak{g} (x)) + \lambda (x) + \phi_\mathfrak{a} (a). 
\end{align}
Since $\phi_\mathfrak{h}$ is a derivation, for any $s(x) + a, ~ s(y) + b \in \mathfrak{h}$, we have
\begin{align*}
\phi_\mathfrak{h} [ s(x) + a, s(y) + b] = [\phi_\mathfrak{h} (s(x) + a), s(y) + b ] + [ s(x) + a , \phi_\mathfrak{h} (s(y) + b)].
\end{align*}
By using the expression (\ref{phi-h}) of $\phi_\mathfrak{h}$, we get from the above equality that
\begin{align*}
s ( \phi_\mathfrak{g} [x, y]) + \lambda ([x, y]) + \phi_\mathfrak{a} (\psi (x, y)) = s ( [\phi_\mathfrak{g} (x), y]) + \psi ( \phi_\mathfrak{g} (x), y ) + s ( [ x, \phi_\mathfrak{g} (y)] ) + \psi (x, \phi_\mathfrak{g} (y)).
\end{align*}
This implies that
\begin{align*}
\phi_\mathfrak{a} ( \psi (x, y)) - \psi ( \phi_\mathfrak{g} (x), y ) - \psi ( x, \phi_\mathfrak{g}(y)) = - \lambda ([x, y]),
\end{align*}
or, equivalently, $\mathrm{Ob}^\mathfrak{h}_{(\phi_\mathfrak{g}, \phi_\mathfrak{a})} = \partial \lambda$ is a coboundary. Hence the obstruction class $[\mathrm{Ob}^\mathfrak{h}_{(\phi_\mathfrak{g}, \phi_\mathfrak{a})} ] \in H^2( \mathfrak{g}, \mathfrak{a})$ is trivial.

To prove the converse part, suppose $\mathrm{Ob}^\mathfrak{h}_{(\phi_\mathfrak{g}, \phi_\mathfrak{a})} $ is given by a coboundary, say $\mathrm{Ob}^\mathfrak{h}_{(\phi_\mathfrak{g}, \phi_\mathfrak{a})}  = \partial \lambda$. We define  a map $\phi_\mathfrak{h} : \mathfrak{h} \rightarrow \mathfrak{h}$ by
\begin{align*}
\phi_\mathfrak{h} ( s(x) + a ) = s ( \phi_\mathfrak{ g}(x) ) + \lambda (x) + \phi_\mathfrak{a} (a).
\end{align*}
Then $\phi_\mathfrak{h}$ is a derivation on $\mathfrak{h}$ and (\ref{new-cent-ex}) is an exact sequence of LeibDer pairs. Hence the pair $(\phi_\mathfrak{g}, \phi_\mathfrak{a})$ is extensible.
\end{proof}

Thus, we obtain the following.

\begin{thm}
If $H^2 (\mathfrak{g}, \mathfrak{a}) = 0$ then any pair of derivations $(\phi_\mathfrak{g}, \phi_\mathfrak{a}) \in \mathrm{Der} (\mathfrak{g}) \times \mathrm{Der} (\mathfrak{a})$ is extensible.
\end{thm}

\subsection{Abelian extensions}\label{subsec-abelian-extension}
Abelian extensions of Leibniz algebras by their representations are studied in \cite{loday-pira}. They are classified (up to equivalence) by the second cohomology group of Leibniz algebras. In this section, we extend this to LeibDer pairs.

Let $(\mathfrak{g}, \phi_\mathfrak{g})$ be a LeibDer pair and $(M, \phi_M)$ be a representation of it. Note that $(M, \phi_M)$ can be considered as a LeibDer pair with the trivial Leibniz bracket on $M$.

\begin{defn}
An abelian extension of $(\mathfrak{g}, \phi_\mathfrak{g})$ by $(M, \phi_M)$ is a short exact sequence of LeibDer pairs
\begin{align}\label{abel-extns}
\xymatrix{
0 \ar[r] & (M, \phi_M ) \ar[r]^i & (\mathfrak{h}, \phi_\mathfrak{h}) \ar[r]^p & ( \mathfrak{g}, \phi_\mathfrak{g}) \ar[r] & 0
}
\end{align}
such that the sequence is split over $\mathbb{K}$ (suppose the splitting is given by a section $s : \mathfrak{g} \rightarrow \mathfrak{h}$ satisfying $p \circ s = \mathrm{id}_\mathfrak{g}$) and the given representation of $(\mathfrak{g}, \phi_\mathfrak{g})$ on $(M, \phi_M)$ is induced from the above extension by $[x, m] = [ s(x), i(m)]$ and $[m, x] = [ i(m), s(x)]$, for $x \in \mathfrak{g}, m \in M$.
\end{defn}

Note that the above induced representation of $(\mathfrak{g}, \phi_\mathfrak{g})$ on $(M, \phi_M)$  does not depend on the choice of the section $s$.

Two abelian extensions $(\mathfrak{h}, \phi_\mathfrak{h}) $ and $(\mathfrak{h}', \phi_{\mathfrak{h}'}) $ are said to be equivalent if there is a LeibDer morphism $\Psi : (\mathfrak{h}, \phi_\mathfrak{h}) \rightarrow (\mathfrak{h}', \phi_{\mathfrak{h}'})$ making the following diagram commutative
\begin{align*}
\xymatrix{
0 \ar[r] & (M, \phi_M ) \ar@{=}[d] \ar[r]^i & (\mathfrak{h}, \phi_\mathfrak{h}) \ar[d]_\Psi \ar[r]^p & ( \mathfrak{g}, \phi_\mathfrak{g}) \ar@{=}[d] \ar[r] & 0 \\
0 \ar[r] & (M, \phi_M ) \ar[r]_{i'} & (\mathfrak{h}', \phi_{\mathfrak{h}'}) \ar[r]_{p'} & ( \mathfrak{g}, \phi_\mathfrak{g}) \ar[r] & 0 .
}
\end{align*}

We denote by $Ext (\mathfrak{g}, M)$ the set of isomorphism classes of abelian extensions of $(\mathfrak{g}, \phi_\mathfrak{g})$ by $(M, \phi_M)$. Then we have the following.

\begin{thm}\label{final-theorem-abelian}
There is a bijection $Ext (\mathfrak{g}, M) \cong H^2_{\mathrm{LeibDer}} (\mathfrak{g}, M).$
\end{thm}

\begin{proof}
Let $(f, \overline{f}) \in C^2_{\mathrm{LeibDer}} (\mathfrak{g}, M)$ be a $2$-cocycle. That is, we have $\delta_L f = 0$ and $\delta_L (\overline{f}) + \delta f = 0$. Consider the direct sum $\mathfrak{h} = \mathfrak{g} \oplus M$ with the bracket
\begin{align*}
[(x,m), (y, n)]_\mathfrak{h} = ( [x, y], [x, n] + [m, y] + f (x, y))
\end{align*}
and a linear map $\phi_\mathfrak{h} (x, m) = ( \phi_\mathfrak{g} (x), \phi_M (m) + \overline{f} (x))$. Then $\mathfrak{h}$ with the above bracket is a Leibniz algebra and $\phi_\mathfrak{h}$ is a derivation for it. More generally, $0 \rightarrow (M, \phi_M ) \xrightarrow{i} (\mathfrak{h}, \phi_\mathfrak{h}) \xrightarrow{p} ( \mathfrak{g}, \phi_\mathfrak{g}) \rightarrow 0 $ is an abelian extension, where $i (m) = (0, m)$ and $p (x, m) = x$.

Let $(f, \overline{f}) - \partial h = ( f - \delta_L h, \overline{f} + \delta h)$ be any $2$-cocycle cohomologous to $(f, \overline{f})$. Consider the corresponding abelian extension $(\mathfrak{h}' = \mathfrak{g} \oplus M, [~, ~]', \phi_{\mathfrak{h}'}).$ They are equivalent via the LeibDer pair morphism $\Psi : \mathfrak{h} \rightarrow \mathfrak{h}'$ given by $\Psi (x, m) = (x, m + h (x))$. Hence the map $H^2_{\mathrm{LeibDer}} (\mathfrak{g}, M) \rightarrow Ext (\mathfrak{g} , M)$ is well-defined.

Conversely, given any abelian extension (\ref{abel-extns}) with splitting $s$, the vector space $\mathfrak{h}$ is isomorphic to $\mathfrak{g} \oplus M$ and $s (x) = (x, 0)$. The maps $i$ and $p$ are the obvious ones. Since $p$ is an algebra map, we have $p [ (x, 0), (y, 0)] = [x, y]$. Hence $[(x, 0), (y, 0)] = ( [x, y], f (x, y)),$ for some $f \in \mathrm{Hom}(\mathfrak{g}^{\otimes 2}, M)$. The Leibniz identity for the bracket on $\mathfrak{h}$ is equivalent to $\delta_L (f) = 0$. Moreover, $p$ commute with derivations, i.e, $p \circ \phi_\mathfrak{h} = \phi_\mathfrak{g} \circ p$. Hence, we have $\phi_\mathfrak{h} (x, 0) = ( \phi_\mathfrak{g} (x), \overline{f} (x))$, for some $\overline{f} \in \mathrm{Hom}(\mathfrak{g}, M)$. The map $\phi_\mathfrak{h}$ is a derivation implies that $\delta_L (\overline{f}) + \delta f = 0$. Hence $(f, \overline{f}) \in C^2_{\mathrm{LeibDer}} (\mathfrak{g}, M)$ is a $2$-cocycle.

If $(\mathfrak{h}, \phi_\mathfrak{h})$ and $(\mathfrak{h}', \phi_{\mathfrak{h}'})$ are two equivalent extensions, then one can show that they are equivalent by a map $\mathfrak{h} = \mathfrak{g} \oplus M \xrightarrow{\Psi} \mathfrak{g} \oplus M = \mathfrak{h}'$ given by $(x, m) \mapsto (x, m + h (x))$, for some $h \in \mathrm{Hom}(\mathfrak{g}, M)$. Let $(f', \overline{f'}) \in C^2_{\mathrm{LeibDer}} (\mathfrak{g}, M)$ be the $2$-cocycle induced from the abelian extension $(\mathfrak{h}', \phi_{\mathfrak{h}'})$. Since $\Psi$ is a morphism of LeibDer pairs, one have $(f, \overline{f}) - ( f', \overline{f'}) = \partial h$. Therefore, the map $Ext (\mathfrak{g}, M) \rightarrow H^2_{\mathrm{LeibDer}} (\mathfrak{g}, M)$ is well-defined. Finally, these two maps are inverses to each other. Hence the proof.
\end{proof}

\section{Deformations of LeibDer pairs}\label{sec-deformation}
In this section, we study formal one-parameter deformations of LeibDer pairs in which we deform both the Leibniz bracket and the distinguished derivation. Our main results in this section are similar to the standard cases.

Let $(\mathfrak{g}, \phi_\mathfrak{g})$ be a LeibDer pair. We denote the Leibniz bracket on $\mathfrak{g}$ by $\mu$, i.e, $\mu (x, y) = [x, y],$ for all $x, y \in \mathfrak{g}$. Consider the space $\mathfrak{g}[[t]]$ of formal power series in $t$ with coefficients from $\mathfrak{g}$. Then $\mathfrak{g}[[t]]$ is a $\mathbb{K}[[t]]$-module.

\begin{defn}
A formal one-parameter deformation of $(\mathfrak{g}, \phi_\mathfrak{g})$ consists of two formal power series
\begin{align*}
\mu_t =~& \sum_{i=0}^\infty t^i \mu_i \in \mathrm{Hom} (\mathfrak{g}^{\otimes 2}, \mathfrak{g})[[t]] ~~~~ \text{ with } \mu_0 = \mu,\\
\phi_t =~& \sum_{i=0}^\infty t^i \phi_i \in \mathrm{Hom} (\mathfrak{g}, \mathfrak{g})[[t]] ~~~~ \text{ with } \phi_0 = \phi_\mathfrak{g}
\end{align*}
such that $\mathfrak{g}[[t]]$ together with the bracket $\mu_t$ forms a Leibniz algebra over $\mathbb{K}[[t]]$ and $\phi_t$ is a derivation on this Leibniz algebra. 
\end{defn}

Thus, in a formal deformation as above, the following identities hold: for $n \geq 0$,
\begin{align}
\sum_{i+j = n} \mu_i (\mu_j (x, y), z) =~& \mu_i (\mu_j (x,z), y) + \sum_{i+j = n} \mu_i (x, \mu_j (y,z)), \\
\sum_{i+j = n} \phi_i (\mu_j (x, y)) =~& \sum_{i+j = n} \mu_i ( \phi_j (x), y) + \mu_i ( x, \phi_j (y)), ~~ \text{ for } x, y, z \in \mathfrak{g}.
\end{align}
Both identities hold for $n=0$ as $(\mathfrak{g}, \phi_\mathfrak{g})$ is a LeibDer pair. However, for $n=1$, we obtain
\begin{align}
\mu_1 ([x,y], z) + [ \mu_1 (x,y), z] =~& \mu_1 ([x, z], y) + [ \mu_1 (x, z), y] + \mu_1 (x, [y,z]) + [x, \mu_1 (y, z)], \label{eq-r}\\
\phi_1 ([x, y]) + \phi_\mathfrak{g} (\mu_1 [x, y]) =~& \mu_1 ( \phi_\mathfrak{g} (x), y) + [\phi_1 (x), y] + \mu_1 (x, \phi_\mathfrak{g} (y)) + [x, \phi_1 (y)],  \label{eq-s}
\end{align}
for all $x, y, z \in \mathfrak{g}$.  The identity (\ref{eq-r}) is equivalent to $\delta_L (\mu_1) = 0$ and the identity (\ref{eq-s}) is equivalent to $\delta_L (\phi_1) + \delta \mu_1 = 0$. Here $\delta_L$ is the coboundary operator of the Leibniz algebra cohomology of $\mathfrak{g}$ with coefficients in itself. Thus, we have
\begin{align*}
\partial (\mu_1, \phi_1) = (\delta_L (\mu_1), \delta_L (\phi_1) + \delta \mu_1 ) = 0.
\end{align*}

\begin{prop}\label{inf-2-cocycle}
Let $(\mu_t, \phi_t)$ be a formal one-parameter deformation of the LeibDer pair $(\mathfrak{g}, \phi_\mathfrak{g})$. Then the linear term $(\mu_1, \phi_1)$ is a $2$-cocycle in the cohomology of the LeibDer pair $(\mathfrak{g}, \phi_\mathfrak{g})$ with coefficients in itself.
\end{prop}

The term $(\mu_1, \phi_1)$ is called the infinitesimal of the deformation. If $(\mu_1, \phi_1) = \cdots = (\mu_{n-1}, \phi_{n-1}) =0$ and $(\mu_n, \phi_n)$ is non-zero, then $(\mu_n, \phi_n)$ is a $2$-cocycle.

\begin{defn}
Two deformations $(\mu_t, \phi_t)$ and $(\mu_t', \phi_t')$ of a LeibDer pair $(\mathfrak{g}, \phi_\mathfrak{g})$ are said to be equivalent if there exists a formal isomorphism $\psi_t = \sum_{i=0}^\infty t^i \psi_i : \mathfrak{g}[[t]] \rightarrow \mathfrak{g} [[t]]$ with $\psi_0 = \mathrm{id}_\mathfrak{g}$ such that
\begin{align*}
\psi_t \circ \mu_t = \mu_t' \circ (\psi_t \otimes \psi_t) \qquad \text{ and } \qquad \psi_t \circ \phi_t = \phi_t' \circ \psi_t.
\end{align*}
\end{defn}
In other words, $\psi_t$ is an isomorphism of LeibDer pairs from  $(\mathfrak{g}[[t]], \mu_t, \phi_t)$ to  $(\mathfrak{g}[[t]], \mu_t', \phi_t')$.
By equating coefficients of $t^n$, we get
\begin{align*}
\sum_{i+j = n} \psi_i \circ \mu_j = \sum_{i+j+k = n} \mu_i' \circ (\psi_j \otimes \psi_k) \qquad \text{ and } \qquad \sum_{i+j = n} \psi_i \circ \phi_j = \sum_{i+j = n} \phi_i' \circ \psi_j.
\end{align*}
The above identities hold for $n=0$, however, for $n=1$, we obtain
\begin{align*}
\mu_1 + \psi_1 \circ \mu = \mu_1' + \mu \circ (\psi_1 \otimes \mathrm{id}) + \mu \circ (\mathrm{id} \otimes \psi_1) ~~~ \text{ ~~and~~ } ~~~
\phi_1 + \psi_1 \circ \phi_\mathfrak{g} = \phi_1' + \phi_\mathfrak{g} \circ \psi_1.
\end{align*}
These two identities together imply that $(\mu_1, \phi_1) - (\mu_1', \phi_1') = \partial (\psi_1).$ Thus, we have the following.

\begin{prop}
The infinitesimals corresponding to equivalent deformations are cohomologous. Hence, they correspond to the same cohomology class.
\end{prop}

\subsection{Rigidity}
\begin{defn}
A formal deformation $(\mu_t, \phi_t)$ of the LeibDer pair $(\mathfrak{g}, \phi_\mathfrak{g})$ is said  to be trivial if it is equivalent to the undeformed one $(\mu_t' = \mu, \phi_t' = \phi_\mathfrak{g}).$
\end{defn}

\begin{defn}
A LeibDer pair $(\mathfrak{g}, \phi_\mathfrak{g})$ is said to be rigid if every formal deformation of $(\mathfrak{g}, \phi_\mathfrak{g})$ is trivial.
\end{defn}

\begin{thm}\label{two-rigidity}
Let $(\mathfrak{g}, \phi_\mathfrak{g})$ be a LeibDer pair. If $H^2_{\mathrm{LeibDer}} (\mathfrak{g}, \mathfrak{g}) = 0$ then $(\mathfrak{g}, \phi_\mathfrak{g})$ is rigid.
\end{thm}

\begin{proof}
Let $(\mu_t, \phi_t)$ be any formal deformation of the LeibDer pair $(\mathfrak{g}, \phi_\mathfrak{g})$. Then by Proposition \ref{inf-2-cocycle} the linear term $(\mu_1, \phi_1)$ is a $2$-cocycle. Thus by the hypothesis, there exists a $1$-cochain $\phi_1 \in C^1_{\mathrm{LeibDer}} (\mathfrak{g}, \mathfrak{g}) = \mathrm{Hom} (\mathfrak{g}, \mathfrak{g})$ such that $(\mu_1, \phi_1) = \partial \psi_1$.

We define a formal isomorphism $\psi_t = \mathrm{id}_\mathfrak{g} + t \psi_1 : \mathfrak{g} [[t]] \rightarrow \mathfrak{g} [[t]]$ and setting
\begin{align}\label{rigid-num}
\mu_t' = \psi_t^{-1} \circ \mu_t \circ (\psi_t \otimes \psi_t)  \qquad \qquad \phi_t' = \psi_t^{-1} \circ \phi_t \circ \psi_t.
\end{align}
Then $(\mu_t', \phi_t')$ is a deformation of the LeibDer pair $(\mathfrak{g}, \phi_\mathfrak{g})$  equivalent to the deformation $(\mu_t, \phi_t)$. It follows from (\ref{rigid-num}) that the deformation $(\mu_t', \phi_t')$ is of the form $\mu_t' = \mu + t^2 \mu_2' + \cdots$ and $\phi_t' = \phi_\mathfrak{g} + t^2 \phi_2' + \cdots $. In other words, the linear terms (coefficients of $t$) of $\mu_t'$ and $\phi_t'$ vanish. One can apply the same argument repeatedly to conclude that $(\mu_t, \phi_t)$ is equivalent to $(\mu, \phi_\mathfrak{g}).$
\end{proof}

\subsection{Finite order deformations and their extensions}
In this subsection, we consider finite order deformations of a LeibDer pair $(\mathfrak{g}, \phi_\mathfrak{g})$. Given a deformation of order $N$, we associate a third cohomology class in the cohomology of the LeibDer pair $(\mathfrak{g}, \phi_\mathfrak{g})$ with coefficients in itself. When this cohomology class is trivial, the order $N$ deformation extends to a deformation of order $N+1$.

\begin{defn}
A deformation of order $N$ of a LeibDer pair $(\mathfrak{g}, \phi_\mathfrak{g})$ consist of finite sums $\mu_t = \sum_{i=0}^N t^i \mu_i$ and $\phi_t = \sum_{i=0}^N t^i \phi_i$ such that $\mu_t$ defines Leibniz bracket on $\mathfrak{g}[[t]]/ (t^{N+1})$ and $\phi_t$ is a derivation on it.
\end{defn}

Thus, the following identities must hold
\begin{align}
 \sum_{i+j = n} \mu_i (\mu_j (x, y), z) =~& \sum_{i+j = n} \mu_i (\mu_j (x,z), y) +  \mu_i (x, \mu_j (y,z)), \\
\sum_{i+j = n} \phi_i (\mu_j (x, y)) =~& \sum_{i+j = n} \mu_i ( \phi_j (x), y) + \mu_i ( x, \phi_j (y)),
\end{align}
for $n=0, 1, \ldots, N$. These identities are equivalent to
\begin{align}
\llbracket \mu, \mu_n \rrbracket =~& -\frac{1}{2} \sum_{i+j=n, i, j > 0} \llbracket \mu_i, \mu_j \rrbracket  \label{n-def-eqn-1}\\
~~~~ \text{ and }~~~ - \llbracket \phi_\mathfrak{g}, \mu_n \rrbracket + \llbracket \mu, \phi_n \rrbracket =~& \sum_{i+j = n, i, j> 0} \llbracket \phi_i, \mu_j \rrbracket. \label{n-def-eqn-2}
\end{align}

\begin{defn}
A deformation $( \mu_t = \sum_{i=0}^N t^i \mu_i, ~\phi_t = \sum_{i=0}^N t^i \phi_i)$ of order $N$ is said to be extensible if there is an element $(\mu_{N+1}, \phi_{N+1}) \in C^2_{\mathrm{LeibDer}} (\mathfrak{g}, \mathfrak{g})$ such that $(\mu_t' = \mu_t + t^{N+1} \mu_{N+1}, ~\phi_t' = \phi_t + t^{N+1}  \phi_{N+1} )$ is a deformation of order $N+1$.
\end{defn}
Thus, two more equations need to be satisfied, namely,
\begin{align*}
 \sum_{i+j = N+1} \mu_i (\mu_j (x, y), z) =~& \sum_{i+j = N+1} \mu_i (\mu_j (x,z), y) +  \mu_i (x, \mu_j (y,z)), \\
\sum_{i+j = N+1} \phi_i (\mu_j (x, y)) =~& \sum_{i+j = N+1} \mu_i ( \phi_j (x), y) + \mu_i ( x, \phi_j (y)).
\end{align*}
The above two equations can be equivalently written as
\begin{align}
\delta_L (\mu_{N+1}) (x, y, x) =~& \frac{1}{2} \sum_{i+j = N+1, i, j > 0} \llbracket \mu_i, \mu_j \rrbracket \quad ( = \mathrm{Ob}^3(x, y, z) ~~~ \text{ say}) \label{eq-el}\\
\delta_L (\phi_{N+1}) + \delta (\mu_{N+1}) (x, y) =~& \sum_{i+j = N+1, i, j > 0} \llbracket \phi_i, \mu_j \rrbracket  \quad ( = \mathrm{Ob}^2(x, y) ~~~ \text{ say}) \label{eq-em}
\end{align}

\begin{prop}\label{ob-three-cocycle}
The pair $( \mathrm{Ob}^3, \mathrm{Ob}^2)  \in C^3_{\mathrm{LeibDer}} (\mathfrak{g}, \mathfrak{g}) $ is a $3$-cocycle in the cohomology of the LeibDer pair $(\mathfrak{g}, \phi_\mathfrak{g})$ with coefficients in itself. 
\end{prop}
\begin{proof}
To prove that $\partial ( \mathrm{Ob}^3, \mathrm{Ob}^2 ) = 0$, it is enough to show that $\delta_L (\mathrm{Ob}^3) = 0$ and $\delta_L (\mathrm{Ob}^2) + (-1)^3 \delta (\mathrm{Ob}^3) = 0$. We have
\begin{align*}
\delta_L (\mathrm{Ob}^3 ) =~& \llbracket \mu, \mathrm{Ob}^3 \rrbracket \\
=~& \frac{1}{2} \sum_{i+j = N+1, i, j > 0} \llbracket \mu , \llbracket \mu_i, \mu_j \rrbracket \rrbracket \\
=~&\frac{1}{2} \sum_{i+j = N+1, i, j > 0} (\llbracket  \llbracket \mu, \mu_i \rrbracket , \mu_j \rrbracket  - \llbracket \mu_i, \llbracket \mu, \mu_j \rrbracket \rrbracket) \\
=~& - \frac{1}{4} \sum_{i' + i'' + j = N+1, i', i'' , j > 0}  \llbracket  \llbracket \mu_{i'}, \mu_{i''} \rrbracket, \mu_j \rrbracket + \frac{1}{4} \sum_{i + j' + j'' = N+1, i, j' , j'' > 0} \llbracket \mu_i , \llbracket \mu_{j'}, \mu_{j''} \rrbracket \rrbracket ~~~ (\text{by } (\ref{n-def-eqn-1}))\\
=~& \frac{1}{4} \sum_{i' + i'' + j = N+1, i', i'' , j > 0} \llbracket \mu_j, \llbracket \mu_{i'} , \mu_{i''} \rrbracket \rrbracket + \frac{1}{4} \sum_{i + j' + j'' = N+1, i, j' , j'' > 0} \llbracket \mu_i , \llbracket \mu_{j'}, \mu_{j''} \rrbracket \rrbracket \\
=~& \frac{1}{2} \sum_{i' + i'' + j = N+1, i', i'' , j > 0} \llbracket \mu_j, \llbracket \mu_{i'} , \mu_{i''} \rrbracket \rrbracket = 0.
\end{align*}
To prove the second part, we observe that
\begin{align*}
\begin{split}
&\delta_{L} (\mathrm{Ob}^2) + (-1)^3~ \delta (\mathrm{Ob}^3) 
= - \llbracket \mu, \mathrm{Ob}^2 \rrbracket + \llbracket \phi_\mathfrak{g}, \mathrm{Ob}^3 \rrbracket \\
&= - \sum_{i+j = N+1, i, j > 0} \llbracket \mu, \llbracket \phi_i, \mu_j \rrbracket \rrbracket + \frac{1}{2} \sum_{i+j = N+1, i, j > 0} \llbracket \phi_\mathfrak{g}, \llbracket \mu_i, \mu_j \rrbracket \llbracket \\
&= - \sum_{i+j = N+1, i, j > 0} \big(  \llbracket \llbracket \mu, \phi_i \rrbracket, \mu_j \rrbracket+ \llbracket \phi_i, \llbracket \mu, \mu_j \rrbracket \rrbracket    \big) + \frac{1}{2} \sum_{i+j = N+1, i, j > 0} \big(  \llbracket \llbracket \phi_\mathfrak{g}, \mu_i \rrbracket, \mu_j \rrbracket +  \llbracket \mu_i, \llbracket \phi_\mathfrak{g}, \mu_j \rrbracket \rrbracket  \big) \\
&= - \sum_{i+j = N+1, i, j > 0} \big(  \llbracket \llbracket \mu, \phi_i \rrbracket, \mu_j \rrbracket + \llbracket \phi_i, \llbracket \mu, \mu_j \rrbracket \rrbracket   \big) + \sum_{i+j = N+1, i, j > 0} \llbracket \llbracket \phi_\mathfrak{g}, \mu_i \rrbracket, \mu_j \rrbracket \\
&= - \sum_{\begin{array}{c}
{i+j = N+1},\\{ i, j > 0}\end{array}
} \big(  \llbracket \llbracket \mu, \phi_i \rrbracket, \mu_j \rrbracket      - \llbracket \llbracket \phi_\mathfrak{g}, \mu_i \rrbracket, \mu_j \rrbracket          ) + \frac{1}{2} \sum_{\begin{array}{c}
{i+ j' + j'' = N+1},\\ {i, j', j'' > 0}\end{array}
}  \llbracket \phi_i, \llbracket \mu_j', \mu_j'' \rrbracket \rrbracket \quad (\text{by } (\ref{n-def-eqn-1})) \\
&= - \sum_{i'+ i'' + j = N+1, i', i'', j > 0} \llbracket \llbracket \phi_{i'}, \mu_{i''} \rrbracket, \mu_j \rrbracket - \sum_{i+j = N+1, i, j >0} ( \llbracket \llbracket \phi_\mathfrak{g}, \mu_i \rrbracket, \mu_j \rrbracket - \llbracket \llbracket \phi_\mathfrak{g}, \mu_i \rrbracket , \mu_j \rrbracket ) \\
&\qquad + \frac{1}{2} \sum_{i+ j' + j'' = N+1, i, j', j'' >0} ( \llbracket \llbracket \phi_i, \mu_{j'} \rrbracket, \mu_{j''} \rrbracket + \llbracket \mu_{j'}, \llbracket \phi_i, \mu_{j''} \rrbracket \rrbracket )   \quad (\text{by } (\ref{n-def-eqn-2}))\\
&= - \sum_{\begin{array}{c}
{i' + i'' + j = N+1},\\{ i', i'', j > 0}\end{array}
} \llbracket \llbracket \phi_{i'}, \mu_{i''} \rrbracket, \mu_j \rrbracket+ \sum_{\begin{array}{c}
{i+j' + j'' = N+1},\\ { i, j', j'' > 0}\end{array}
} \llbracket \llbracket  \phi_i, \mu_{j'}\rrbracket, \mu_{j''} \rrbracket = 0.
\end{split}
\end{align*}
Hence the proof.
\end{proof}

Note that the right hand sides of (\ref{eq-el}) and (\ref{eq-em}) does not involve $\mu_{N+1}$ or $\phi_{N+1}$. Therefore, the cohomology class $[( \mathrm{Ob}^3, \mathrm{Ob}^2)]  \in H^3_{\mathrm{LeibDer}} (\mathfrak{g}, \mathfrak{g}) $ is basically obtained from the order $N$ deformation $(\mu_t, \phi_t)$. This class is called the obstruction class to extend the deformation. If this class is vanishes, i.e, $( \mathrm{Ob}^3, \mathrm{Ob}^2)$ is given by a coboundary, then
\begin{align*}
( \mathrm{Ob}^3, \mathrm{Ob}^2) = \partial (\mu_{N+1}, \phi_{N+1}) = ( \delta_L (\mu_{N+1}), \delta_L (\phi_{N+1}) + \delta (\mu_{N+1})),
\end{align*}
for some $(\mu_{N+1}, \phi_{N+1}) \in C^2_{\mathrm{LeibDer}} (\mathfrak{g}, \mathfrak{g}).$ Then it follows from the above observation that $(\mu_t' = \mu_t + t^{N+1} \mu_{N+1}, \phi_t' = \phi_t + t^{N+1} \phi_{N+1})$ is a deformation of order $N+1$. In other words, $(\mu_t, \phi_t)$ is extensible. On the other hand, if $(\mu_t, \phi_t)$ is extensible then $( \mathrm{Ob}^3, \mathrm{Ob}^2)$ is given by the coboundary $\partial (\mu_{N+1}, \phi_{N+1}).$ Hence the corresponding obstruction class vanishes.

\begin{thm}\label{ob-three-cocycle-thm}
A deformation  $(\mu_t, \phi_t)$ of order $N$ is extensible if and only if the obstruction class $[( \mathrm{Ob}^3, \mathrm{Ob}^2)] \in H^3_{\mathrm{LeibDer}} (\mathfrak{g}, \mathfrak{g})$  is vanishes.
\end{thm}

\begin{thm}
If $H^3_{\mathrm{LeibDer}} (\mathfrak{g}, \mathfrak{g}) = 0$ then every finite order deformation of $(\mathfrak{g}, \phi_\mathfrak{g})$ is extensible.
\end{thm}

\begin{corollary}
If $H^3_{\mathrm{LeibDer}} (\mathfrak{g}, \mathfrak{g}) = 0$ then every $2$-cocycle in the cohomology of the LeibDer pair $(\mathfrak{g}, \phi_\mathfrak{g})$ with coefficients in itself is the infinitesimal of a formal deformation of $(\mathfrak{g}, \phi_\mathfrak{g}).$
\end{corollary}
\section{Homotopy derivations on sh Leibniz algebras}\label{sec-homo-der}
The notion of sh Leibniz algebras was introduced in \cite{ammar-poncin}. Here we will mostly emphasis on sh Leibniz algebras whose underlying graded vector space is concentrated only in two degrees, namely $0$ and $1$. Such sh Leibniz algebras are called $2$-term sh Leibniz algebras. We define homotopy derivations on $2$-term sh Leibniz algebras. Finally, we classify skeletal and strict $2$-term sh
Leibniz algebras with homotopy derivations.

\begin{defn}
A $2$-term sh Leibniz algebra consists of a chain complex $A_1 \xrightarrow{d} A_0$ together with bilinear maps $l_2 : A_i \times A_j \rightarrow A_{i+j}$, for $ 0 \leq i, j, i+j \leq 1$ and a trilinear map $l_3 : A_0 \times A_0 \times A_0 \rightarrow A_1$ satisfying the following identities: for $x, y, z, w \in A_0$ and $m, n \in A_1$,
\begin{itemize}
\item $d l_2 (x, m) = l_2 (x, dm), $
\item $d l_2 (m, x) = l_2 (dm, x),$
\item $l_2 (dm, n) = l_2 (m, dn),$
\item $d l_3 (x, y, z) =  l_2 ( l_2 (x, y), z) - l_2 ( l_2 (x, z), y) - l_2 ( x, l_2 (y, z)),$
\item $l_3 (x, y, dm) =  l_2 ( l_2 (x, y), m) - l_2 ( l_2 (x, m), y) - l_2 ( x, l_2 (y, m)),$
\item $ l_3 (x, dm, y) =  l_2 ( l_2 (x, m), y ) - l_2 ( l_2 (x, y), m) - l_2 ( x, l_2 (m, y)),$
\item $ l_3 (dm, x, y) =  l_2 ( l_2 (m, x), y ) - l_2 ( l_2 (m, y), x) - l_2 ( m, l_2 (x, y)),$ 
\item $l_2 (x, l_3 (y,z,w)) + l_2 ( l_3 (x,z,w), y) - l_2 ( l_3 (x, y, w), z) + l_2 ( l_3 (x, y,z), w) \\= l_3 ( l_2 (x,y), z, w) - l_3 ( l_2 (x, z), y, w) + l_3 ( l_2 (x, w), y, z) \\
 - l_3 (x, l_2 (y, z), w) + l_3 (x, l_2 (y, w), z) + l_3 ( x, y, l_2 (z, w)).$
\end{itemize}
\end{defn}
A $2$-term sh Leibniz algebra as above may be denoted by $(A_1 \xrightarrow{d} A_0, l_2, l_3)$. When $A_1 = 0$, one simply get a Leibniz algebra structure on $A_0$ with the bracket given by $l_2 : A_0 \times A_0 \rightarrow A_0$.

A $2$-term sh Leibniz algebra $(A_1 \xrightarrow{d} A_0, l_2, l_3)$ is said to be skeletal if the differential $d = 0.$
Skeletal algebras are in one-to-one correspondence with tuples $(\mathfrak{g}, M , \theta)$ where $\mathfrak{g}$ is a Leibniz algebra, $M$ is a representation of $\mathfrak{g}$ and $\theta \in C^3(\mathfrak{g}, M)$ is a $3$-cocycle in the Leibniz algebra cohomology of $\mathfrak{g}$ with coefficients in $M$ \cite{sheng-liu}.
 More precisely, let $(A_1 \xrightarrow{0} A_0, l_2, l_3)$ be a skeletal algebra. Then $(A_0, l_2)$ is a Leibniz algebra; $A_1$ is a representation of it with left and right actions given by $[x,m]:= l_2 (x, m)$ and $[m,x]:= l_2 (m, x)$, for $x \in A_0, m \in A_1$. Finally, the map $l_3 : A_0 \times A_0 \times A_0 \rightarrow A_1$ is a $3$-cocycle in the cohomology of $A_0$ with coefficients in $A_1$.
 
 \begin{defn}
Let  $(A_1 \xrightarrow{d} A_0, l_2, l_3)$ and  $(A_1' \xrightarrow{d'} A_0', l_2', l_3')$ be $2$-term sh Leibniz algebras. A morphism between them consists of a chain map of underlying chain complexes (i.e, linear maps $f_0 : A_0 \rightarrow A_0'$ and $f_1 : A_1 \rightarrow A_1'$ satisfying $d' \circ f_1 = f_0 \circ d$) and a bilinear map $f_2 : A_0 \times A_0 \rightarrow A_1'$ satisfying
\begin{itemize}
\item $d ( f_2 (x, y)) = f_0 ( l_2 (x, y)) - l_2' ( f_0 (x), f_0 (y)),$
\item $f_2 ( x, dm) = f_1 ( l_2 (x, m)) - l_2' ( f_0 (x), f_1 (m)$,
\item $f_2 ( dm, x) = f_1 ( l_2 (m,x)) - l_2' ( f_1 (m), f_0 (x)),$
\item $f_1 ( l_3 (x,y,z)) + l_2' ( f_0 (x, y), f_0 (z)) - l_2' ( f_2 (x,z), f_0 (y)) - l_2' ( f_0 (x), f_2 (y,z)) \\
+ f_2 ( l_2 (x,y), z) - f_2 ( l_2 (x, z), y) - f_2 ( x, l_2 (y, z)) - l_3' ( f_0 (x), f_0 (y), f_0 (z))= 0$,
\end{itemize}
for $x, y, z \in A_0$ and $m \in A_1$.
\end{defn}

We denote the category of $2$-term sh Leibniz algebras and morphisms between them by ${\bf 2Leib}_\infty$.
 
 \begin{defn}\label{homo-deri}
 Let $(A_1 \xrightarrow{d} A_0, l_2, l_3)$ be a $2$-term sh Leibniz algebra. A homotopy derivation on it consists of a chain map of the underlying chain complex (i.e, linear maps $\theta_0 : A_0 \rightarrow A_0$ and $\theta_1 : A_1 \rightarrow A_1$ satisfying $d \circ \theta_1 = \theta_0 \circ d$) and a bilinear map $\theta_2 : A_0 \times A_0 \rightarrow A_1$ satisfying the following: for $x, y, z \in A_0$ and $m \in A_1$,
 \begin{itemize}
 \item[(a)] $d \theta_2 (x, y) = \theta_0 ( l_2 (x,y)) - l_2 ( \theta_0 (x), y) - l_2 (x, \theta_0 (y)),$
 \item[(b)] $\theta_2 (x, dm) = \theta_1 ( l_2 (x, m)) - l_2 ( \theta_0 (x), m) - l_2 (x, \theta_1 (m)),$
 \item[(c)] $\theta_2 (dm, x) = \theta_1 (l_2 (m, x)) - l_2 ( \theta_1 (m), x) - l_2 ( m, \theta_0 (x)),$
 \item[(d)] $l_3 ( \theta_0 (x), y,z) + l_3 (x, \theta_0 (y), z) + l_3 (x, y, \theta_0 (z)) - \theta_1 l_3 (x, y, z) \\
 = l_2 ( \theta_2 (x, y), z) - l_2 ( \theta_2 (x, z), y) - l_2 (x, \theta_2 (y,z)) 
 + \theta_2 (l_2 (x, y), z) - \theta_2 ( l_2(x, z), y) - \theta_2 (x, l_2 (y, z)).$
 \end{itemize}
 \end{defn}
 
 A $2$-term sh Leibniz algebra with a homotopy derivation as above may be denoted by the pair $((A_1 \xrightarrow{d} A_0, l_2, l_3), (\theta_0, \theta_1, \theta_2)).$ Such a pair is called a $2\mathrm{LeibDer}_\infty$ pair.
 
 \begin{defn}
Let $((A_1 \xrightarrow{d} A_0, l_2, l_3), (\theta_0, \theta_1, \theta_2))$ and
 $((A_1' \xrightarrow{d'} A_0', l_2', l_3'), (\theta_0', \theta_1', \theta_2'))$ be $2\mathrm{LeibDer}_\infty$ pairs. A morphism between them consists of a morphism $(f_0, f_1, f_2)$ between the underlying $2$-term sh Leibniz algebras and a linear map $\mathcal{B}: A_0 \rightarrow A_1'$ satisfying
 \begin{itemize}
\item $f_0 ( \theta_0 (x)) - \theta_0' ( f_0 (x)) = d' (\mathcal{B}(x)),$
\item $f_1 (\theta_1 (m)) - \theta_1' (f_1(m)) = \mathcal{B} (dm),$
\item $f_1 (\theta_2 (x, y)) - \theta_2' ( f_0 (x), f_0 (y)) = \theta_1' ( f_2 (x, y)) - f_2 ( \theta_0 (x), y) - f_2 (x, \theta_0 (y)) \\
+ \mathcal{B} (l_2 (x, y)) - \mu_2' ( \mathcal{B}(x), f_0 (y)) - \mu_2' ( f_0 (x), \mathcal{B}(y)).$
\end{itemize}  
\end{defn}

We denote the category of $2\mathrm{LeibDer}_\infty$ pairs and morphisms between them by ${\bf 2LeibDer}_\infty$ pair.
 
 \begin{prop}\label{skeletal-prop}
 There is a one-to-one correspondence between skeletal $2$-term sh Leibniz algebras with homotopy derivations and triples $((\mathfrak{g}, \phi_\mathfrak{g}), (M, \phi_M), (\theta, \overline{\theta}))$ where $(\mathfrak{g}, \phi_\mathfrak{g})$ is a LeibDer pair, $(M, \phi_M)$ is a representation and $(\theta, \overline{\theta}) \in C^3_{\mathrm{LeibDer}} (\mathfrak{g}, M)$ is a $3$-cocycle in the cohomology of the LeibDer pair $(\mathfrak{g}, \phi_\mathfrak{g})$ with coefficients in $(M, \phi_M).$
 \end{prop}
 
 \begin{proof}
 Let $((A_1 \xrightarrow{0} A_0, l_2, l_3), (\theta_0, \theta_1, \theta_2))$ be a skeletal $2$-term sh Leibniz algebra with a homotopy derivation. Then it follows from Definition \ref{homo-deri} (a) that $\theta_0$ is a derivation for the Leibniz algebra $(A_0, l_2)$. The conditions (b) and (c) of Definition \ref{homo-deri} says that the pair $(A_1, \theta_1)$ is a representation of the LeibDer pair $(A_0, \theta_0)$. Finally, the condition (d) implies that $\delta_L (\theta_2 ) + \delta (l_3) = 0$. Thus $(l_3, - \theta_2 ) \in C^3_{\mathrm{LeibDer}} ( A_0, A_1)$ is a $3$-cocycle in the cohomology of the LeibDer pair $(A_0, \theta_0)$ with coefficients in the representation $(A_1, \theta_1)$.
 
 For the converse part, let $((\mathfrak{g}, \phi_\mathfrak{g}), (M, \phi_M), (\theta, \overline{\theta}))$ be such a triple. Then it is easy to see that $((M \xrightarrow{0} \mathfrak{g}, l_2= [~,~], l_3 = \theta), (\phi_\mathfrak{g}, \phi_M, - \overline{\theta} ))$ is a skeletal $2$-term sh Leibniz algebra. These two correspondences are inverses to each other.
 \end{proof}
 
 A $2$-term sh Leibniz algebra $(A_1 \xrightarrow{d} A_0, l_2, l_3)$ is called strict if $l_3 =0$. Further, a homotopy derivation $(\theta_0, \theta_1, \theta_2)$ on it is said to be strict if $\theta_2 = 0$.
 
 Crossed module of Leibniz algebras was introduced in \cite{loday-pira}. See also \cite{sheng-liu}. Here we extend this notion by introducing crossed module of LeibDer pairs.
 
 \begin{defn}\label{crossed-defn}
 A crossed module of LeibDer pairs consists of a quadruple $( (\mathfrak{g}, \phi_\mathfrak{g}), (\mathfrak{h}, \phi_\mathfrak{h}), dt, \phi )$ in which $(\mathfrak{g}, \phi_\mathfrak{g}), (\mathfrak{h}, \phi_\mathfrak{h}) $ are both LeibDer pairs, $dt : \mathfrak{g} \rightarrow \mathfrak{h}$ is a morphism of LeibDer pairs and 
 \begin{align*}
 \phi: \mathfrak{h} \times \mathfrak{g} \rightarrow \mathfrak{g} \qquad \qquad \phi : \mathfrak{g} \times \mathfrak{h} \rightarrow \mathfrak{g}
\end{align*}  
defines a representation of the LeibDer pair $(\mathfrak{h}, \phi_\mathfrak{h})$ on $(\mathfrak{g}, \phi_\mathfrak{g})$ satisfying the following conditions: for $m, n \in \mathfrak{g}$ and $x \in \mathfrak{h}$,
\begin{itemize}
\item[(a)] $dt ( \phi (x, m)) = [ x, dt (m)]$,\\
$dt ( \phi (m, x)) = [ dt(m), x],$
\item[(b)] $[ dt (m), n ] = [m, n],$\\
$[m, dt (n)] = [m, n],$
\item[(c)] $\phi ([m,n], x) = [\phi (m, x), n] + [m, \phi (n, x)],$\\
$ [ \phi (x, m), n] = [ \phi (x, n), m] + \phi (x, [m,n]),$\\
$[\phi (m, x), n] = \phi ([m,n], x) + [m, \phi (x, n)],$
\item[(d)] $\phi_\mathfrak{g} ( \phi (x, m)) = \phi ( \phi_\mathfrak{h}(x), m) + \phi (x, \phi_\mathfrak{g} (m),$\\
$\phi_\mathfrak{g} (\phi (m, x)) = \phi ( \phi_\mathfrak{g} (m), x) + \phi (m, \phi_\mathfrak{h} (x)).$
\end{itemize}
 \end{defn}
 
 When $\phi_\mathfrak{g} = 0, \phi_\mathfrak{h} = 0$ one simply get crossed module of Leibniz algebras. Therefore, the condition (d) of the above definition is absent in a crossed module of Leibniz algebras.
 
 \begin{thm}\label{thm-crossed-leibder}
 There is a one-to-one correspondence between strict $2$-term sh Leibniz algebras with strict homotopy derivations and crossed module of LeibDer pairs.
 \end{thm}
 
 \begin{proof}
 It has been proved in \cite{sheng-liu} that strict $2$-term sh Leibniz algebras are in one-to-one correspondence with crossed module of Leibniz algebras. More precisely, $(A_1 \xrightarrow{d} A_0, l_2, l_3= 0)$ is a strict $2$-term sh Leibniz algebra if and only if $(A_1, A_0, d, l_2)$ is a crossed module of Leibniz algebras, where the Leibniz algebra structure on $A_1$ and $A_0$ are respectively given by $[m,n]:= l_2 (dm, n) = l_2 (m, dn)$ and $[x, y] = l_2 (x, y)$, for $m, n \in A_1$ and $x, y \in A_0$. It follows from Definition \ref{homo-deri} that $\theta_1$ is a derivation on the Leibniz algebra $A_1$ and $\theta_0$ is a derivation on the Leibniz algebra $A_0$. In other words $(A_1, \theta_1)$ and $(A_0, \theta_0)$ are both LeibDer pairs. Since $\theta_0 \circ d = d \circ \theta_1$, the map $dt = d : A_1 \rightarrow A_0$ is a morphism of LeibDer pairs. Finally, the conditions (b) and (c) of Definition \ref{homo-deri} are equivalent to the last condition of Definition \ref{crossed-defn}. Hence the proof.
 \end{proof}

\section{Categorification of LeibDer pairs}\label{sec-2-der}
Leibniz $2$-algebras are categorification of Leibniz algebras \cite{sheng-liu}. In this section, we introduce categorified derivations (also called $2$-derivations) on Leibniz $2$-algebras.

Let {\bf Vect} denote the category of vector spaces. A $2$-vector space is a category internal to the
category {\bf Vect}. Thus, a $2$-vector space $\mathbb{V}$ is a category with a vector space of objects $V_0$ and a vector space of morphisms $V_1$ such that all structure maps are linear. Let $s, t : V_1 \rightarrow V_0$ be the source and target maps respectively.  Given a $2$-vector space $\mathbb{V} = ( V_1 \rightrightarrows V_0)$, we have a $2$-term chain complex $\text{ker}(s) \xrightarrow{t} V_0$. Conversely, any $2$-term chain complex $A_1 \xrightarrow{d} A_0$ gives rise to a $2$-vector space $\mathbb{A} = (A_0 \oplus A_1 \rightrightarrows A_0)$ with space of objects $A_0$ and space of morphisms given by $A_0 \oplus A_1$; the structure maps are given by $s ( x \oplus m) = x, ~ t ( x \oplus m) = x + dm$, for $x \in A_0$ and $m \in A_1$. It has been shown in \cite{baez-crans} that the category of $2$-vector spaces and the category of $2$-term chain complexes are equivalent.

\begin{defn}
A Leibniz $2$-algebra is a $2$-vector space $\mathbb{V}$ with a bilinear functor $[~, ~] : \mathbb{V} \times \mathbb{V} \rightarrow \mathbb{V}$ and a trilinear natural isomorphism
\begin{align*}
J_{x, y, z} : [[x, y], z] \rightarrow [[x, z], y] + [x, [y, z]] , ~~~ \text{ for } x, y, z \in {V_0}
\end{align*}
commuting the following diagram
\[
\xymatrix{
[[[x,y],z], w] \ar[r]^{J_{[x,y], z, w}} \ar[d]_{J_{x, y,z}} & [[[x, y], w], z] + [[x, y], [z, w]] \ar[d]^{J_{x, y, w} + 1 }\\
[[[x,z], y] + [x, [y, z]], w]  \ar[d]_{J_{[x,z], y, w} + J_{x, [y,z], w}} & [[[x,w], y] + [x, [y,w]], z] + [[x, y], [z, w]]  \ar[d]^{J_{[x,w],y,z} + J_{x, [y,w], z} + J_{x, y, [z, w]} } \\
[[[x,z], w], y] + [[x, z], [y, w]] + [[x, w], [y, z]] + [x, [[y,z], w]] \ar[r]_{J_{x,z, w} + 1+ 1+ J_{y,z, w}} &  P 
}
\]
where $P = [[[x,w], z], y] + [[x, [z,w]], y] + [[x, z], [y, w]] + [[x, w], [y, z]] + [x, [[y, w], z]] + [x, [y, [z,w]]].$
\end{defn}

\begin{defn}
Let $(\mathbb{V}, [~, ~], J)$ and $(\mathbb{V}', [~, ~]', J')$ be two Leibniz $2$-algebras. A morphism between them consists of a linear functor $F = (F_0, F_1)$ from the underlying vector space $\mathbb{V}$ to $\mathbb{V}'$, a natural transformation
\begin{align*}
F_2 (x, y) : [ F_0 (x), F_0 (y)]' \rightarrow F_0 [x, y]
\end{align*}
such that the following diagram commute
\[
\xymatrix{
[[ F_0 (x), F_0 (y)]', F_0 (z)]' \ar[r]^{J'} \ar[d]_{[F_2, 1]'} & [[F_0(x), F_0 (z)]', F_0 (y)]' + [F_0 (x), [F_0 (y), F_0 (z)]']' \ar[d]^{[F_2, 1]' + [1, F_2]'} \\
[ F_0 [x,y], F_0 (z)]' \ar[d]_{F_2} & [ F_0 [x, z] , F_0 (y)]' + [ F_0 (x), F_0 [y,z] ]' \ar[d]^{F_2 + F_2} \\
F_0 [[x,y],z] \ar[r]_{F_0 (J)} & F_0 ([[x,z], y] + [x, [y,z]]).
}
\]
\end{defn}

Leibniz $2$-algebras and morphisms between them forms a category. We denote this category by ${\bf Leib2}$. 

In the next, we define $2$-derivations on Leibniz $2$-algebras. They are categorification of derivations on Leibniz algebras.

\begin{defn}
Let $(\mathbb{V}, [~,~], J)$ be a Leibniz $2$-algebra. A $2$-derivation on it consists of a linear functor $D : \mathbb{V} \rightarrow \mathbb{V}$ and a natural isomorphism
\begin{align*}
\mathcal{D}_{x, y} : D[x, y] \rightarrow [Dx, y] + [x, Dy], ~~~ \text{ for } x, y \in {V}_0
\end{align*}
making the following diagram commutative
\[
\xymatrix{
D[[x,y],z] \ar[r]^J \ar[d]_{\mathcal{D}_{[x,y],z}} & D ( [[x,z], y] + [x, [y,z]]) \ar[d] \\
[D[x,y], z] + [[x,y], Dz] \ar[d]_{[\mathcal{D}, 1] + 1} & [D[x,z], y] + [[x,z], Dy] + [Dx, [y,z]] + [x, D[y,z]] \ar[d]^{[\mathcal{D},1]+1+1+[1,\mathcal{D}]} \\
[  [Dx, y] + [x, Dy], z ] + [[x,y], Dz] \ar[r]_{J + J + J} & Q 
}
\]
where $Q = [[Dx,z], y] + [[x, Dz], y] + [[x,z], Dy] + [Dx, [y,z]] + [x, [Dy,z]] + [x, [y, Dz]]$.
\end{defn}

We call a Leibniz $2$-algebra with a $2$-derivation, a LeibDer$2$ pair.

\begin{defn}
Let $(\mathbb{V}, [~,~], J, D, \mathcal{D})$ and $(\mathbb{V}', [~,~]', J', D', \mathcal{D}')$ be two LeibDer$2$ pairs. A morphism between them consists of a Leibniz $2$-algebra morphism $(F = (F_0, F_1), F_2)$ and a natural isomorphism
\begin{align*}
\Phi_x : D' (F_0 (x)) \rightarrow F_0 (D(x)), ~~~ \text{ for } x \in {V}_0
\end{align*} 
which makes the following diagram commutative
\[
\xymatrix{
D' ([F_0 (x), F_0 (y)]') \ar[r]^{F_2} \ar[d]_{\mathcal{D}'} & D' (F_0 [x,y]) \ar[d]^{\Phi_{[x,y]}} \\
[D' (F_0 (x)) , F_0 (y)]' + [F_0 (x), D' (F_0 (y))]'  \ar[d]_{[\Phi_x, 1]' + [1, \Phi_y]'} & F_0 (D[x,y])  \ar[d]^{\mathcal{D}}\\
[F_0 (D(x)), F_0 (y)]' + [F_0 (x), F_0 (D(y))]' \ar[r]_{F_2 + F_2} & F_0 ([Dx, y] + [x, Dy]). 
}
\]
\end{defn}
We denote the category of LeibDer$2$ pairs and morphisms between them by {\bf LeibDer2}.

It is shown in \cite{sheng-liu} that the category ${\bf Leib2}$ is equivalent to the category ${\bf 2Leib}_\infty$. This generalizes the similar theorem for the case of Lie algebras \cite{baez-crans}. Let us recall the construction of a $2$-term sh Leibniz algebra from a Leibniz $2$-algebra and a Leibniz $2$-algebra from a $2$-term sh Leibniz algebra \cite{sheng-liu}.

Let $(\mathbb{V}, [~, ~], J)$ be a Leibniz $2$-algebra. Then $(\mathrm{ker}(s) \xrightarrow{t} V_0, l_2, l_3)$ is a $2$-term sh Leibniz algebra where
\begin{align*}
~&l_2 (x, y) := [x,y], ~~~~ l_2 (x, m) := [x, m], ~~~~ l_2 (m,x) := [m,x], ~~~~ l_2 (m,n):= 0  \text{ and }\\
~&l_3 (x, y, z) := \mathrm{pr}(J_{x, y,z }),
\end{align*} 
for $x, y \in V_0, ~m , n \in \mathrm{ker} (s)$ and $\mathrm{pr}$ denote the projection on $\mathrm{ker}(s).$

Conversely, given a $2$-term sh Leibniz algebra $(A_1 \xrightarrow{d} A_0, l_2, l_3)$, the corresponding Leibniz $2$-algebra is defined on the $2$-vector space $\mathbb{A} = ( A_0 \oplus A_1 \rightrightarrows A_0)$ with
\begin{align*}
[x \oplus m, y \oplus n] :=~& l_2 (x, y) \oplus l_2 (x, n) + l_2 (m, y) + l_2 (m, dn),\\
J_{x, y, z} :=~& ( [[x, y], z], l_3 (x, y, z)).
\end{align*}

\begin{thm}\label{thm-equiv-cat}
The categories {\bf LeibDer2} and ${\bf 2LeibDer}_\infty$ are equivalent.
\end{thm}

\begin{proof}
Here we only sketch the construction of a homotopy derivation from a $2$-derivation and vice versa. The rest of the verifications are similar to \cite[Theorem 4.3.6]{baez-crans}.

Given a LeibDer$2$ pair $(\mathbb{V}, [~,~], J, D, \mathcal{D})$, we define a homotopy derivation on the $2$-term sh Leibniz algebra $(\mathrm{ker}(s) \xrightarrow{t} V_0, l_2, l_3)$ by
\begin{align*}
\theta_0 (x) := D (i(m)), \quad \theta_1 (m) := D|_{\mathrm{ker}(s)} (m) \quad \text{ and } \quad \theta_2 (x, y) := \mathrm{pr} (\mathcal{D}_{x,y}).
\end{align*}
If $(F_0, F_1, F_2, \Phi)$ is a morphism of LeibDer$2$ pairs, then $(f_0 = F_0, f_1 = F_1 |_{\mathrm{ker}(s)}, f_2 = \mathrm{pr} \circ F_2, \mathcal{B} = \Phi)$ is a morphism  between corresponding $2$-term sh Leibniz algebras with homotopy derivations.

Conversely, let $( (A_1 \xrightarrow{d} A_0, l_2, l_3), (\theta_0, \theta_1, \theta_2))$ be a $2$-term sh Leibniz algebra with a homotopy derivation. We define a $2$-derivation $(D, \mathcal{D}')$ on the Leibniz $2$-algebra $(A_0 \oplus A_1 \rightrightarrows A_0, [~,~], J)$ by
\begin{align*}
D(x,m) := (\theta_0 (x), \theta_1 (m)) \quad \text{ and } \quad \mathcal{D}_{x, y} := ( [x,y], \theta_2 (x,y)).
\end{align*}
If $(f_0, f_1, f_2, \mathcal{B})$ is a morphism of $2$-term sh Leibniz algebras with homotopy derivations, then $(F_0 = f_0, F_1 = f_1, F_2(x,y) = ( [f_0 (x) , f_0 (y)]' , f_2 (x,y)), \Phi = \mathcal{B})$ is a morphism between corresponding LeibDer$2$ pairs.
\end{proof}

\section{Conclusions}

In this paper, we consider LeibDer pairs as noncommutative analog of LieDer pairs. We study their (central and abelian) extensions and deformations from cohomological point of view. We define homotopy derivations on sh Leibniz algebras and relate them with categorification of LeibDer pairs.

In \cite{bala1, bala2} the author considered extensions and deformations of algebras over a binary quadratic operad $\mathcal{P}$. The results of the present paper can be extended to $\mathcal{P}$-algebras with derivations. In \cite{das-prep} we plan to systematically study extensions and deformations of a pair of a $\mathcal{P}$-algebra and a derivation on it.

Leibniz algebras play an important role in the study of Courant algebroids \cite{wein-cou}. In \cite{roth, ji} the authors study deformation theory of Courant algebroids from algebraic and Poisson geometric point of view. It might be interesting to explore the importance of derivations in a Courant algebroid and one may extend the results of \cite{roth, ji} in the context of Courant algebroids with derivations.

\medskip

\noindent {\bf Acknowledgements.} The author would like to thank Indian Institute of Technology Kanpur for financial support.

\end{document}